\documentclass[10pt]{article}

\usepackage{amsmath}
\usepackage{amsfonts}
\usepackage{mathrsfs}
\usepackage{amssymb, color}

\usepackage[linkcolor=black,anchorcolor=black,citecolor=black]{hyperref}
\usepackage{graphicx}
\usepackage[body={15.5cm,21cm}, top=3cm]{geometry}%
\usepackage[pagewise,displaymath,mathlines]{lineno}
\providecommand{\U}[1]{\protect \rule{.1in}{.1in}}
\newtheorem{theorem}{Theorem}[section]

\newtheorem{corollary}[theorem]{Corollary}

\newtheorem{definition}[theorem]{Definition}

\newtheorem{lemma}[theorem]{Lemma}

\newtheorem{proposition}[theorem]{Proposition}
\newtheorem{remark}[theorem]{Remark}

\newenvironment{proof}[1][Proof]{\noindent \textbf{#1.} }{\  \rule{0.5em}{0.5em}}
\begin{document}

	\title{Multi-dimensional BSDEs driven by $G$-Brownian motion and related system of fully nonlinear PDEs}
	\author{Guomin Liu \thanks{School of Mathematical Sciences,
			Fudan University, Shanghai, 200433, PR China; and Zhongtai Securities Institute for Financial Studies, Shandong University, Jinan,
			250100, PR China. gmliusdu@163.com. Research supported by the National Natural Science Foundation of
			China (No. 11601282 and No.11671231) and the Natural Science Foundation of Shandong Province (No. ZR2016AQ10). }}
		\date{}
	\maketitle
	
	\begin{abstract}
		In this paper, we study the  well-posedness of multi-dimensional backward stochastic differential equations driven by $G$-Brownian motion ($G$-BSDEs). The existence and uniqueness of solutions are obtained via a contraction argument for $Y$ component and a backward iteration of local solutions. Furthermore, we show that, the solution of
		multi-dimensional  $G$-BSDE in a Markovian framework provides a probabilistic formula for the viscosity
		solution of a system of nonlinear parabolic partial differential equations.
	\end{abstract}
	
	
	\textbf{MSC-classification}: 60H10, 60H30.
	\section{Introduction}

	The nonlinear backward stochastic differential equation driven by Brownian motion (BSDE)  was  introduced by Pardoux and Peng \cite{PP}.
	A solution to a BSDE is a couple of processes $(Y,Z)$ satisfying
	\[Y_t=\xi+\int_t^T f(s,Y_s,Z_s)ds-\int_t^T Z_sdB_s, \ 0\leq t\leq T,\]
	where  $B_t$ is a $d$-dimensional standard Brownian motion and the function $f$ is
	called the generator of BSDE. Since then, it was shown that the BSDE theory
	is a powerful tool for the study of partial differential equations,  stochastic control, theoretical economics
	and mathematical finance, see \cite{BH,CHMP,EPQ,P93}.
	
	Recently, motivated by financial problems with Knightian uncertainty, Peng \cite{P3,P4,P10}
	has established systematically  a framework of time-consistent sublinear expectation,
	called $G$-expectation.
	Under this framework, a new type of nonlinear Brownian motion $(B_t)_{t\geq 0}$ called $G$-Brownian motion and the corresponding  stochastic calculus were established.
	Furthermore, Hu, Ji, Peng and Song \cite{HJPS} obtained the existence and uniqueness of
	solutions to the following one-dimensional  (i.e., $Y$ is one-dimensional)
	backward stochastic differential equations driven by $G$-Brownian motion ($G$-BSDEs):
	\begin{equation}
	Y_t=\xi+\int_t^Tf(s,Y_s,Z_s)ds+\int_t^T g_{ij}(s,Y_s,Z_s)d\langle B^i,B^j\rangle_s-\int_t^T Z_s dB_s-(K_T-K_t),
	\end{equation}
	where the solution of this equation consists of a triple of processes $(Y,Z,K)$.  Note that, compared with the classical case, $G$-BSDE has an extra non-increasing $G$-martingale term $K$.
	In their
	accompanying paper \cite{HJPS1}, the corresponding comparison theorem, Feymann-Kac formula were established.
	For other developments, one can refer to, for instance, \cite{HW2,LPH,Song3}. See also \cite{STZ} for a formulation highly related to $G$-BSDEs, called 2BSDEs.
	
	Multi-dimensional $G$-BSDE refers to the case that $Y$ in the solution is multi-dimensional. Since $G$-expectation is a nonlinear expectation, the linear combination of $G$-martingales is not a $G$-martingale anymore. This leads to the difficulty that the method for one-dimensional $G$-BSDE in \cite{HJPS} can not be applied to solving multi-dimensional $G$-BSDEs.

	We denote by $x^l$ the $l$-th column (component) of $x \in\mathbb{R}^n$, for $l=1,\cdots,n$. In the paper, we consider the following multi-dimensional $G$-BSDEs:
	\begin{equation}\label{eq0.1}
	Y_{t}^{l} =\xi^{l}+\int^{T}_{t} f^{l}(s,Y_{s},Z_{s}^{l})ds+\int^{T}_{t}
	g_{ij}^{l}(s,Y_{s},Z_{s}^{l})d\langle B^{i},B^{j}\rangle_{s}-\int^{T}_{t}
	Z_{s}^{l} dB_{s}-(K_{T}^{l}-K_{t}^{l}),\ 1\leq l\leq n.
	\end{equation}
	Here the generators $f,g_{ij}$ are assumed to be diagonal in $z$, i.e., the $z$ parts of the $l$-th components $f^l,g^l_{ij}$ only depend on $z^l$.
	The existence and uniqueness theorem (see Theorem \ref{my16}) of the local solution to (\ref{eq0.1}) is proved by the method of Picard iteration for the $Y$ term, in which we have applied the well-posedness theorem of one-dimensional $G$-BSDEs and  the tools of a priori estimates. The global solution (see Theorem \ref{my17}) is then obtained by a backward iteration of the local ones.
	The Picard iteration for local solutions is also used to establish the comparison theorem (see Theorem \ref{Myth4.1}) for multi-dimensional $G$-BSDE (\ref{eq0.1}) through a limit argument.

	 In the last part of our paper, we study the link with the viscosity
	 solution of a system of fully nonlinear parabolic partial differential equations  in a Markovian framework. Indeed, \cite{HJPS1} showed that the one-dimensional $G$-BSDE is connected to fully nonlinear parabolic partial differential equation. On the other hand, the solution
	 to a standard multi-dimensional BSDE provides the viscosity solution to a system of semilinear parabolic partial differential
	 equations, see, e.g., Pardoux and Peng \cite{PP2} Pardoux, Pradeilles and Rao \cite{PPR}, Peng \cite{Peng92}. Inspired by these two directions, we show that, our
	 multi-dimensional  $G$-BSDE can provide a probabilistic representation for the
	 solution of a system of fully nonlinear parabolic partial differential equations (see Theorem \ref{Myth4.3}). For this purpose, we derive some useful a priori estimates of continuous dependence of solutions on terminal conditions  and generators for  $G$-BSDEs (\ref{eq0.1}), and also study the regularity of solutions of $G$-BSDEs (\ref{eq0.1}) coupled with stochastic differential equations driven by $G$-Brownian motion ($G$-SDEs).
	
	This paper is organized as follows. In Section 2, we recall some basic notions
	of  $G$-Brownian motion and one-dimensional $G$-BSDEs. Section 3 is devoted to
	the well-posedness of multi-dimensional $G$-BSDEs and the  corresponding comparison theorem. Finally, in Section 4, we give its relationship with the system of PDEs.

	\section{Preliminaries}
	
	The main purpose of this section is to recall some basic results about
	$G$-expectation and one-dimensional $G$-BSDEs, which are needed in the sequel.
	The readers may refer to \cite{DHP11,HJPS,HJPS1,P10} for more details. For
	convenience, every element $x \in\mathbb{R}^{n}$ is identified to a column
	vector with $l$-th component $x^{l}$ and the corresponding Euclidian norm and
	Euclidian scalar product is denoted by $|\cdot|$ and $\langle\cdot
	,\cdot\rangle$, respectively.
	
	\subsection{$G$-expectation}
	
	Let $\Omega_{T}:=C_{0}([0,T];\mathbb{R}^{d})$ be the space of all
	$\mathbb{R}^{d} $-valued continuous paths $(\omega_{t})_{t\geq0}$ starting
	from origin,  equipped with the supremum norm.  Denote by $\mathcal{B}%
	(\Omega_{T})$ the Borel $\sigma$-algebra of $\Omega_{T}$ and $B_{t}%
	(\omega):=\omega_{t}$ the canonical mapping.  For each $t\in\lbrack0,T]$, we
	set
	\[
	L_{ip}(\Omega_{T}):=\{ \varphi(B_{t_{1}},\ldots,B_{t_{k}}):k\in\mathbb{N}
	,t_{1},\ldots,t_{k}\in\lbrack0,T],\varphi\in C_{b.Lip}(\mathbb{R} ^{k\times d
	})\},
	\]
	where $C_{b.Lip}(\mathbb{R}^{k\times d})$ denotes the space of  bounded and
	Lipschitz functions on $\mathbb{R}^{k\times d}$.
	
	Let $\mathbb{S}(d)$ be the space of all $d\times d$ symmetric matrices. For
	each given monotonic and sublinear function $G:\mathbb{S} (d)\rightarrow
	\mathbb{R}$, Peng \cite{P10} constructed a dynamic coherent nonlinear
	$G$-expectation space $(\Omega_{T}, L_{ip}(\Omega_{T}),\hat{\mathbb{E}},
	(\hat{\mathbb{E}}_{t})_{t\geq0})$.  In this paper, we  always assume that $G$
	is non-degenerate, i.e., there exists a constant  $\underline{\sigma}^{2}>0$
	such that
	\[
	G(A)-G(B)\geq\frac{1}{2}\underline{\sigma}^{2}\text{tr}[A-B], \text{ for
	}A\geq B.
	\]

	For each  $p\geq1$, the completion of $L_{ip}(\Omega_{T})$ under the norm
	$||X||_{L_{G}^{p}}:=(\mathbb{\hat{E}}[|X|^{p}])^{1/p}$ is denoted by
	$L_{G}^{p}(\Omega_{T})$. The $G$-expectation $\mathbb{\hat{E}}[\cdot]$ and
	conditional $G$-expectation $\mathbb{\hat{E}}_{t}[\cdot]$ can be  extended
	continuously to the completion $L_{G}^{p}(\Omega_{T})$. The corresponding
	canonical process $B$  is called $G$-Brownian motion.
	
	Indeed, the $G$-expectation can be regarded as an upper expectation on
	$L_{G}^{1}(\Omega_{T})$.
	
	\begin{theorem}
		[\cite{DHP11, HP09}] \label{the1.1} There exists a weakly compact set
		$\mathcal{P}$ of probability  measures on $(\Omega_{T},\mathcal{B}(\Omega
		_{T}))$ such that
		\[
		\hat{\mathbb{E}}[\xi]=\sup_{P\in\mathcal{P}}E_{P}[\xi]=\max_{P\in\mathcal{P}%
		}E_{P}[\xi],\ \ \ \ \text{for all}\ \xi\in{L}_{G}^{1}{(\Omega_{T})}.
		\]
		
	\end{theorem}
	
	For this $\mathcal{P}$, we define capacity
	\[
	c(A):=\sup_{P\in\mathcal{P}}P(A),\ A\in\mathcal{B}(\Omega_{T}).
	\]
	A set $A\subset\Omega_{T}$ is polar if $c(A)=0$. A property holds
	\textquotedblleft quasi-surely\textquotedblright\ (q.s. for short) if it
	holds  outside a polar set. In the following, we do not distinguish between
	two random  variables $X$ and $Y$ if $X=Y$ q.s..
	
	For $p\geq1$, we set
	\[
	\mathbb{L}^{p}(\Omega_{T}):=\{X\in\mathcal{B}(\Omega_{T}):\sup_{P\in
		\mathcal{P}}E_{P}[|X|^{p}]<\infty\}.
	\]
	It is important to note that $L_{G}^{p}(\Omega_{T})\subset\mathbb{L}
	^{p}(\Omega_{T})$. We extend $G$-expectation $\mathbb{\hat{E}}$ to
	$\mathbb{L}^{p}(\Omega_{T})$, and still denote it by $\mathbb{\hat{E}}$, by
	setting
	\[
	\mathbb{\hat{E}}[X]:=\sup_{P\in\mathcal{P}}E_{P}[X], \ \text{for each}
	\ X\in\mathbb{L}^{1}(\Omega_{T}).
	\]
	For $p\geq1$, $\mathbb{L}^{p}(\Omega_{T})$ is a Banach space under the norm
	$(\mathbb{\hat{E}}[|\cdot|^{p}])^{1/p}$.
	
	In the sequel, we shall also use the $G$-evaluation introduced in
	\cite{song1}. For $\xi\in L_{ip}(\Omega_{T})$, we define its $G$-evaluation by
	$\mathcal{E}[\xi]:=\mathbb{\hat{E}} [\sup_{t\in\lbrack0,T]}\mathbb{\hat{E}%
	}_{t}[\xi]]$.  For $p\geq1$ and $\xi\in L_{ip}(\Omega_{T})$, define $\Vert
	\xi\Vert_{p, \mathcal{E}}=\{ \mathcal{E}[|\xi|^{p}]\}^{1/p}$ and denote by
	$L_{\mathcal{E} }^{p}(\Omega_{T})$ the completion of $L_{ip}(\Omega_{T})$
	under the norm $ \Vert\cdot\Vert_{p,\mathcal{E}}$.
	
	\begin{theorem}
		\label{the2.10} For any $\alpha\geq1$ and $\delta>0$, we  have $L_{G}%
		^{\alpha+\delta}(\Omega_{T})\subset L_{\mathcal{E} }^{\alpha}(\Omega_{T})$.
		More precisely, for all $\xi\in L_{ip}(\Omega_{T})$, we  have
		\begin{equation}
		\label{Myeq2.2}\mathbb{\hat{E}}[\sup_{t\in\lbrack0,T]}\mathbb{\hat{E}}%
		_{t}[|\xi|^{\alpha}]]\leq C(\mathbb{\hat{E}}[|\xi|^{\alpha+\delta}%
		])^{\alpha/(\alpha+\delta)},
		\end{equation}
		for some constant $C$ depending only on $\alpha$ and $\delta$.
	\end{theorem}
	
	\begin{remark}
		\upshape{
			In \cite{song1}, the author gives the following inequality for the estimate of $G$-evaluation: there exists some constant $C$ depending only on $\alpha$ and $\delta$ such that
			\begin{equation*}
			\mathbb{\hat{E}}[\sup_{t\in \lbrack0,T]}\mathbb{\hat{E}}_{t}[|\xi|^{\alpha
			}]]\leq C\{(\mathbb{\hat{E}}[|\xi|^{\alpha+\delta}])^{\alpha
				/(\alpha+\delta)}+\mathbb{\hat{E}}[|\xi|^{\alpha+\delta}]\}.
			\end{equation*}
			The refree points out that the above inequality can be refined to be (\ref{Myeq2.2}), which is more similar to the classical case. Indeed,
			multiplying $\xi$ by $\lambda>0$ and making $\lambda$ tend to $0$, then the nonhomogeneous term is eliminated and we get the stronger inequality (\ref{Myeq2.2}).
		}
	\end{remark}
	
	\subsection{One-dimensional $G$-BSDEs}
	
	First, we shall introduce the corresponding stochastic calculus with respect
	to $G$-Brownian motion.
	
	\begin{definition}
		\label{def2.6} Let $M_{G}^{0}(0,T)$ be the collection of processes in the
		following form: for a given partition $\{t_{0},\cdot\cdot\cdot,t_{N}\}$ of
		$[0,T]$,
		\[
		\eta_{t}(\omega)=\sum_{j=0}^{N-1}\xi_{j}(\omega)I_{[t_{j},t_{j+1})}(t),
		\]
		where $\xi_{i}\in L_{ip}(\Omega_{t_{i}})$, $i=0,1,2,\cdot\cdot\cdot,N-1$. For
		$p\geq1$ and $\eta\in M_{G}^{0}(0,T)$, let $\Vert\eta\Vert_{H_{G}^{p}}=\{
		\mathbb{\hat{E}}[(\int_{0}^{T}|\eta_{s}|^{2}ds)^{p/2}]\}^{1/p}$, $\Vert
		\eta\Vert_{M_{G}^{p}}=\{ \mathbb{\hat{E}}[\int_{0}^{T}|\eta_{s}|^{p}
		ds]\}^{1/p}$ and denote by $H_{G}^{p}(0,T)$, $M_{G}^{p}(0,T)$ the completions
		of $M_{G}^{0}(0,T)$ under the norms $\Vert\cdot\Vert_{H_{G}^{p}}$, $\Vert
		\cdot\Vert_{M_{G}^{p}}$ respectively.
	\end{definition}
	
	For each $1\leq i, j \leq d$, we denote by $\langle B^{i}, B^{j}\rangle$ the
	mutual variation process.  Then for two processes $\eta\in H_{G}^{p}(0,T)$ and
	$\xi\in M_{G}^{p}(0,T)$ with $p\geq1$,  the $G$-It\^{o} integrals $\int%
	^{\cdot}_{0}\eta_{s}dB^{i}_{s}$ and $\int^{\cdot}_{0}\xi_{s}d\langle
	B^{i},B^{j}\rangle_{s}$ are well defined, see \cite{HJPS}.
	
	Now we introduce some basic results of one-dimensional $G$-BSDEs. Consider the
	following type of $G$-BSDEs (in this paper we always use Einstein's summation
	convention)
	\begin{equation}
	\label{eq1.1}Y_{t}=\xi+\int_{t}^{T} f(s,Y_{s},Z_{s})ds+\int_{t}^{T}
	g_{ij}(s,Y_{s},Z_{s})d\langle B^{i},B^{j}\rangle_{s}-\int_{t}^{T} Z_{s}
	dB_{s}-(K_{T}-K_{t}),
	\end{equation}
	where
	\[
	f(t,\omega,y,z),g_{ij}(t,\omega,y,z):[0,T]\times\Omega_{T}\times
	\mathbb{R}\times\mathbb{R}^{d}\rightarrow\mathbb{R},
	\]
	satisfy the following properties:
	
	\begin{description}
		\item[($A_{1}$)] there exists some $\beta>1$ such that for any $y,z$,
		$f(\cdot,\cdot,y,z),g_{ij}(\cdot,\cdot,y,z)\in M_{G}^{\beta}(0,T)$,
		
		\item[($A_{2}$)] there exists some $L>0$ such that
		\[
		|f(t,y,z)-f(t,y^{\prime},z^{\prime})|+\sum_{i,j=1}^{d}|g_{ij}(t,y,z)-g_{ij}%
		(t,y^{\prime},z^{\prime})|\leq L(|y-y^{\prime}|+|z-z^{\prime}|).
		\]
		
	\end{description}
	
	Let $S_{G}^{0}(0,T)=\{h(t,B_{t_{1}\wedge t}, \ldots,B_{t_{n}\wedge t}%
	):t_{1},\ldots,t_{n}\in[0,T],h\in C_{b,Lip}(\mathbb{R}^{n+1})\}$. For $p\geq1$
	and $\eta\in S_{G}^{0}(0,T)$, set $\|\eta\|_{S_{G}^{p}}=\{\hat{\mathbb{E}%
	}[\sup_{t\in[0,T]}|\eta_{t}|^{p}]\}^{1/p}$. Denote by $S_{G}^{p}(0,T)$ the
	completion of $S_{G}^{0}(0,T)$ under the norm $\|\cdot\|_{S_{G}^{p}}$.  For
	simplicity, we denote by $\mathcal{A}_{G}^{p}(0,T)$ the collection of
	processes $K$ such that $K$ is a non-increasing $G$-martingale with $K_{0}=0$
	and $K_{T}\in L_{G}^{p}(\Omega_{T})$.
	
	In the rest of this paper, for simplicity we denote by $M^{p}_{G}
	(0,T;\mathbb{R}^{n})$ the set of $n$-dimensional stochastic process
	$X=(X^{1},\cdots,X^{n})$ such that $X^{l}\in M^{p}_{G}(0,T),l\leq n$, and we
	also  define ${S}_{G}^{\alpha}(0,T;\mathbb{R}^{n})$, $H^{\alpha}%
	_{G}(0,T;\mathbb{R} ^{n\times d})$, $\mathcal{A}^{\alpha}_{G}(0,T;\mathbb{R}%
	^{n})$ and $L^{\beta}_{G}(\Omega_{T};\mathbb{R}^{n})$ similarly.
	
	\begin{theorem}
		[Theorem 4.1 of \cite{HJPS}]\label{wellposeness of G-BSDE}  Assume that
		$\xi\in L_{G}^{\beta}(\Omega_{T})$ and $f,g_{ij}$ satisfy $(A_{1})$-$(A_{2})$
		for some $\beta>1$. Then for any $1<\alpha<\beta$, equation \eqref{eq1.1} has
		a unique solution $(Y,Z,K)\in{S}_{G}^{\alpha}(0,T)\times H^{\alpha}%
		_{G}(0,T;\mathbb{R}^{d})\times\mathcal{A}_{G}^{\alpha}(0,T)$.
	\end{theorem}
	
	\begin{theorem}
		[Comparison theorem (\cite{HJPS1})]\label{the1.5}  Let $(Y_{t}^{(l)}%
		,Z_{t}^{(l)},K_{t}^{(l)})_{t\leq T}$, $l=1,2$, be the solutions of the
		following $G$-BSDEs:
		\[
		Y^{(l)}_{t}=\xi^{(l)}+\int_{t}^{T} f^{(l)}(s,Y^{(l)}_{s},Z^{(l)}_{s}%
		)ds+\int_{t}^{T} g^{(l)}_{ij}(s,Y^{(l)}_{s},Z^{(l)}_{s})d\langle B^{i}%
		,B^{j}\rangle_{s}-\int_{t}^{T} Z^{(l)}_{s} dB_{s}-(K^{(l)}_{T}-K^{(l)}_{t}),
		\]
		where $f^{(l)},\ g^{(l)}_{ij}$ satisfy $(A_{1})$-$(A_{2})$ and $\xi^{(l)}\in
		L_{G}^{\beta}(\Omega_{T})$ with $\beta>1$. If $\xi^{(1)}\geq\xi^{(2)}$,
		$f^{(1)}\geq f^{(2)}$, $[g^{(1)}_{ij}]_{i,j=1}^{d}\geq[g^{(2)}_{ij}%
		]_{i,j=1}^{d}$ $($here $[g^{(k)}_{ij}]_{i,j=1}^{d}$ stands for $d\times d$ matrix with entries $g^{(k)}_{ij}$, $k=1,2$, and  $[g^{(1)}_{ij}]_{i,j=1}^{d}\geq[g^{(2)}_{ij}%
		]_{i,j=1}^{d}$ means that $[g^{(1)}_{ij}]_{i,j=1}^{d}-[g^{(2)}_{ij}%
		]_{i,j=1}^{d}$ is nonnegative definite$)$, then $Y_{t}^{(1)}\geq Y_{t}^{(2)}$.
	\end{theorem}
	
	We also have the following estimates.
	
	\begin{proposition}
		[Propositions 3.8 and 5.1 of \cite{HJPS}] \label{pro3.5} For each $T\leq S$,
		let $\xi^{(l)}\in L_{G}^{\beta}(\Omega_{T})$,  $l=1,2$, and $f^{(l)}$,
		$g_{ij}^{(l)}$ satisfy $(A_{1})$ and $(A_{2})$ for some $\beta>1,L>0$.
		Assume that $(Y^{(l)},Z^{(l)},K^{(l)})\in{S}_{G}^{\alpha}(0,T)\times
		H^{\alpha}_{G}(0,T;\mathbb{R}^{d})\times\mathcal{A}_{G}^{\alpha}(0,T)$, for
		some $1<\alpha\leq\beta$, are the solutions of equation \eqref{eq1.1}
		corresponding to  $\xi^{(l)}$, $f^{(l)}$ and $g_{ij}^{(l)}$. Set $\hat{Y}%
		_{t}=Y_{t}^{(1)}-Y_{t}^{(2)}$.  Then there exist a constant $C(\alpha)$
		depending on  $S$, $G$, $L$ and $\alpha$ such that
		\begin{equation}
		|\hat{Y}_{t}|^{\alpha} \leq C(\alpha)\mathbb{\hat{E}}_{t}[|\hat{\xi}|^{\alpha
		}+(\int_{t} ^{T}\hat{h}_{s}ds)^{\alpha}],\label{my101}%
		\end{equation}
		where $\hat{\xi}=\xi^{(1)}-\xi^{(2)}$ and $\hat{h}_{s}=|f^{(1)}(s,Y_{s}%
		^{(2)},Z_{s} ^{(2)})-f^{(2)}(s,Y_{s}^{(2)},Z_{s}^{(2)})|+\sum_{i,j=1}%
		^{d}|g_{ij}^{(1)}(s,Y_{s} ^{(2)},Z_{s}^{(2)})-g_{ij}^{(2)}(s,Y_{s}^{(2)}%
		,Z_{s}^{(2)})|$.
	\end{proposition}
	
	\begin{corollary}
		[Corollary 5.2 of \cite{HJPS}]\label{pro3.6}  For each $T\leq S$, let $\xi\in
		L_{G}^{\beta}(\Omega_{T})$ and $f$, $g_{ij}$ satisfy $(A_{1})$ and $(A_{2})$
		for some $\beta>1,L>0$.  Assume that $(Y,Z,K)\in{S}_{G}^{\alpha}(0,T)\times
		H^{\alpha}_{G}(0,T;\mathbb{R}^{d})\times\mathcal{A}_{G}^{\alpha}(0,T)$, for
		some $1<\alpha\leq\beta$, is a solution of equation \eqref{eq1.1}
		corresponding to  $\xi$, $f$ and $g_{ij}$.  Then there exist a constant
		$C(\alpha)$ depending on  $S$, $G$, $L$ and $\alpha$ such that
		\begin{equation}
		\label{eq2.5}|{Y}_{t}|^{\alpha} \leq C(\alpha)\mathbb{\hat{E}}_{t}[|{\xi
		}|^{\alpha}+(\int_{t} ^{T}h_{s}ds)^{\alpha}],
		\end{equation}
		where $h_{s}=|f(s,0,0)|+\sum_{i,j=1}^{d}|g_{ij}(s,0,0)|$.
	\end{corollary}
	
	\begin{proposition}
		[Propositions 3.8 of \cite{HJPS}] \label{pro2.9} Let $\xi^{(l)}\in
		L_{G}^{\beta}(\Omega_{T})$,  $l=1,2$, and $f^{(l)}$, $g_{ij}^{(l)}$ satisfy
		${(A_{1})}$ and ${(A_{2})}$ for some $\beta>1,L>0$.  Assume that
		$(Y^{(l)},Z^{(l)},K^{(l)})\in{S}_{G}^{\alpha}(0,T)\times H^{\alpha}%
		_{G}(0,T;\mathbb{R}^{d})\times\mathcal{A}_{G}^{\alpha}(0,T)$, for some
		$1<\alpha\leq\beta$, are the solutions of equation \eqref{eq1.1} corresponding
		to  $\xi^{(l)}$, $f^{(l)}$ and $g_{ij}^{(l)}$. Set $\hat{Y}_{t}=Y_{t}%
		^{(1)}-Y_{t}^{(2)} ,\hat{Z}_{t}=Z_{t}^{(1)}-Z_{t}^{(2)}$.  Then there exist a
		constant $C(\alpha)$ depending on  $T$, $G$, $L$ and $\alpha$ such that
		\begin{equation}
		\mathbb{\hat{E}}[(\int_{0}^{T}|\hat{Z}_{s}|^{2}ds)^{\frac{\alpha}{2}}]\leq
		C(\alpha)\{ \Vert\hat{Y}\Vert_{{S}^{\alpha}_{G}}^{\alpha}+\Vert\hat{Y}
		\Vert_{{S}^{\alpha}_{G}}^{\frac{\alpha}{2}}\sum_{l=1}^{2}[||Y^{(l)}||_{
			{S}^{\alpha}_{G}}^{\frac{\alpha}{2}}+||\int_{0}^{T}h_{s}^{(l)}ds||_{
			L^{\alpha}_{G}}^{\frac{\alpha}{2}}]\},\label{my102}%
		\end{equation}
		where $h_{s}^{(l)}=|f^{(l)}(s,0,0)|+\sum_{i,j=1}^{d}|g^{(l)}_{ij}(s,0,0)|$.
	\end{proposition}
	
	\section{Well-posedness of multi-dimensional $G$-BSDEs}
	
	In this paper, we shall consider the following type of $n$-dimensional
	$G$-BSDE  on the interval $[0,T]$:
	\begin{equation}
	\label{my1}
	\begin{split}
	Y_{t}^{l}  &  =\xi^{l}+\int^{T}_{t} f^{l}(s,Y_{s},Z_{s}^{l})ds+\int^{T}_{t}
	g_{ij}^{l}(s,Y_{s},Z_{s}^{l})d\langle B^{i},B^{j}\rangle_{s}-\int^{T}_{t}
	Z_{s}^{l} dB_{s}-(K_{T}^{l}-K_{t}^{l}),\ 1\leq l\leq n,
	\end{split}
	\end{equation}
	where
	\[
	f^{l}(t,\omega,y,z^{l}), g_{ij}^{l}(t,\omega,y,z^{l}): [0,T]\times\Omega
	_{T}\times\mathbb{R}^{n}\times\mathbb{R}^{d}\rightarrow\mathbb{R},
	\ \ \forall1\leq l\leq n,
	\]
	satisfy:
	
	\begin{description}

		\item[($H_{1}$)] there is some constant $\beta>1$ such that for each $y,z$,
		$f^{l}(\cdot,\cdot,y,z^{l}),g_{ij}^{l}(\cdot,\cdot,y,z^{l})\in M_{G}^{\beta
		}(0,T)$,
		
		\item[($H_{2}$)] there exists some $L>0$ such that, for each $y_{1},y_{2}%
		\in\mathbb{R}^{n}, z^{l}_{1},z^{l}_{2}\in\mathbb{R}^{d},$
		\[
		|f^{l}(t,y_{1},z^{l}_{1})-f^{l}(t,y_{2},z_{2}^{l})|+\sum_{i,j=1}^{d}
		|g_{ij}^{l}(t,y_{1},z^{l}_{1})-g_{ij}^{l}(t,y_{2},z_{2}^{l})|\leq
		L(|y_{1}-y_{2}|+|z_{1}^{l}-z^{l}_{2}|).
		\]
		
	\end{description}
	
	In the sequel we denote by $M^{p} _{G}(a,b;\mathbb{R}^{n})$, $S_{G}^{\alpha
	}(a,b;\mathbb{R}^{n})$, $H^{\alpha}_{G}(a,b;\mathbb{R}^{n\times d})$ and
	$\mathcal{A}^{\alpha}_{G} (a,b;\mathbb{R}^{n})$ the corresponding spaces for
	the stochastic processes  having time indexes on $[a,b]$.
	
	\subsection{Existence and Uniqueness}
	
	We first study the local solution to $G$-BSDE \eqref{my1}. Indeed, we have
	
	\begin{theorem}
		\label{my16} Assume that $(H_{1})-(H_{2})$ hold for some $\beta>1$. Then
		there exists a constant $0<\delta\leq T$ depending only on $T,G,n, \beta$  and
		$L$ such that for any $h\in(0,\delta]$, $t\in[0,T-h]$ and given $\zeta\in
		L^{\beta} _{G}(\Omega_{t+h};\mathbb{R}^{n})$, the $G$-BSDE on the interval
		$[t,t+h]$
		\begin{equation}
		\label{my1-1}%
		\begin{split}
		Y_{s}^{l}  &  =\zeta^{l}+\int^{t+h}_{s} f^{l}(r,Y_{r},Z_{r}^{l})dr+\int%
		^{t+h}_{s} g_{ij}^{l}(r,Y_{r},Z_{r}^{l})d\langle B^{i},B^{j}\rangle_{r}
		-\int^{t+h}_{s} Z_{r}^{l} dB_{r}-(K_{t+h}^{l}-K_{s}^{l}),\\
		&  \ \ \ \ 1\leq l\leq n,
		\end{split}
		\end{equation}
		admits a unique solution $(Y,Z,K)\in{S}_{G}^{\alpha}(t,t+h;\mathbb{R}
		^{n})\times H^{\alpha}_{G}(t,t+h;\mathbb{R}^{n\times d})\times\mathcal{A}
		^{\alpha}_{G}(t,t+h;\mathbb{R}^{n})$ for each $1<\alpha<\beta$. Moreover,
		$Y\in{M}_{G}^{\beta}(t,t+h;\mathbb{R}^{n}). $
	\end{theorem}
	
	In order to prove Theorem \ref{my16}, we consider the following $G$-BSDE on
	the interval $[t,t+h]$:
	\begin{equation}%
	\begin{split}
	Y_{s}^{U,l}  &  =\zeta^{l}+\int_{s}^{t+h}f^{l,U}(r,Y_{r}^{U,l},Z_{r}%
	^{U,l})dr+\int_{s}^{t+h}g_{ij}^{l,U}(r,Y_{r}^{U,l},Z_{r}^{U,l})d\langle
	B^{i},B^{j}\rangle_{r}-\int_{s}^{t+h}Z_{r}^{U,l}dB_{r}\\
	&  \ \ \ -(K_{t+h}^{U,l}-K_{s}^{U,l}),\ \ \ 1\leq l\leq n,
	\end{split}
	\label{myq4}%
	\end{equation}
	where $U\in{M}_{G}^{\beta}(t,t+h;\mathbb{R}^{n})$, $\zeta\in L_{G}^{\beta
	}(\Omega_{t+h};\mathbb{R}^{n})$, $h\in\lbrack0,T-t]$ and
	\[
	\psi^{l,U}(t,y^{l},z^{l})=\psi^{l}(t,U_{t}^{1},\cdots,U_{t}^{l-1},y^{l}%
	,U_{t}^{l+1},\cdots,U_{t}^{n},z^{l}):[0,T]\times\Omega_{T}\times
	\mathbb{R}\times\mathbb{R}^{d}\rightarrow\mathbb{R},\ \ \text{ for
		$\psi=f,g_{ij}$}.
	\]
	Denote $X^{U}=(X^{U,1},\cdots,X^{U,n})$ for $X=Y,Z,K$. Then
	
	\begin{lemma}
		\label{myq6} Suppose assumptions $(H_{1})-(H_{2})$ hold for some $\beta>1$.
		Then, for any $1<\alpha<\beta$, the $G$-BSDE \eqref{myq4} has a unique solution $(Y^{U},Z^{U},K^{U})$ in
		${S}_{G}^{\alpha}(t,t+h;\mathbb{R}^{n})\times H_{G}^{\alpha}(t,t+h;\mathbb{R}
		^{n\times d})\times\mathcal{A}_{G}^{\alpha}(t,t+h;\mathbb{R}^{n})$. Moreover, $Y^{U}\in{M}_{G}^{\beta}(t,t+h;\mathbb{R}
		^{n})$.
	\end{lemma}
	
	\begin{proof}
		For each $l\geq1$, with the help of assumptions $(H_{1}),(H_{2})$ and Lemma
		\ref{MG lemma} below, we have $f^{l,U}(s,y^{l},z^{l}),g^{l,U}_{ij}
		(s,y^{l},z^{l})\in M^{\beta}_{G}(t,t+h)$ for each $y^{l}\in\mathbb{R}
		,z^{l}\in\mathbb{R}^{d}$. Then applying Theorem \ref{wellposeness of G-BSDE}
		to components of $G$-BSDE (\ref{myq4}) yields the existence and uniqueness
		result of solutions to \eqref{myq4}.
		
		Now we shall prove the second part. Let $1\leq l\leq n$ be given. By
		Corollary  \ref{pro3.6}, we can find some constant $C^{\star}$ depending on
		$T,G,\beta$  and $L$ such that
		\[
		|Y_{s}^{U,l}|^{\beta}\leq C^{\star}\mathbb{\hat{E}}_{s}[|{\zeta}^{l}|^{\beta
		}+(\int_{s}^{t+h}|{h}_{r}^{l,0}|dr)^{\beta}]:=\rho_{s},\ \forall s\in\lbrack
		t,t+h],
		\]
		where $h_{s}^{l,0}=|f^{l,U}(s,0,0)|+|g_{ij}^{l,U}(s,0,0)|$. Note that $\eta\in
		L_{G}^{1}(\Omega_{t+h})$ implies $\mathbb{\hat{E}}_{s}[|\eta|]\in M_{G}
		^{1}(t,t+h)$ from a standard approximation argument. Then we have $\rho_{s}\in
		M_{G}^{1}(t,t+h).$ Now recalling Theorem 4.7 in \cite{HW1}, we conclude that
		\[
		\lim\limits_{N\rightarrow\infty}\mathbb{\hat{E}}[\int_{t}^{t+h}|Y_{s}
		^{U,l}|^{\beta}|\mathbf{1}_{\{|Y_{s}^{l}|^{\beta}\geq N\}}]\leq\lim
		\limits_{N\rightarrow\infty}\mathbb{\hat{E}}[\int_{t}^{t+h}\rho_{s}
		\mathbf{1}_{\{\rho_{s}\geq N\}}]=0,
		\]
		which, together with the fact that $(Y_{s}^{U,l})_{t\leq s\leq t+h}\in{S}
		_{G}^{\alpha}(t,t+h)\subset{M}_{G}^{\alpha}(t,t+h)$, indicates that
		$Y^{U,l}\in{M}_{G}^{\beta}(t,t+h).$ The proof is complete.
	\end{proof}
	
	\begin{lemma}
		\label{MG lemma} Suppose $p\geq1$. Given a function $h:[t,t+h]\times
		\Omega_{T}  \times\mathbb{R}^{m}\rightarrow\mathbb{R}$ satisfying
		$h(\cdot,\cdot,x)\in M^{p}_{G}(t,t+h)$ for each $x\in\mathbb{R}^{m}$. Assume
		that $h$ is Lipschitz  continuous in $x$ with some coefficient $L$, uniformly
		in $(t,\omega)$. Then  for any $X\in M^{p}_{G}(t,t+h;\mathbb{R}^{m})$, we have
		$(h(s,X_{s}))_{t\leq s\leq t+h}\in M^{p}_{G}(t,t+h)$.
	\end{lemma}
	
	\begin{proof}
		According to the partition of unit theorem, for each integer $k\geq1$, we can
		find a sequence of continuous functions $\{\phi_{i} ^{k}\}_{i=1}^{\infty}$
		such that the diameter of support $\lambda($supp$(\phi_{i}^{k}))\leq1/k$ and
		$0\leq\phi_{i}^{k} \leq1$. Moreover, for any integer $N\geq1$, there exists a
		constant $k_{N}$ such that
		\[
		I_{B_{N}(0)}(x)\leq\sum_{i=1}^{k_{N}}\phi_{i}^{k}(x)\leq1,
		\]
		where $B_{N}(0)$ denotes the open ball centered at $0$ with radius $N$.
		Choose $x_{i}^{k}\in\mathbb{R}^{m}$ such that $\phi_{i}^{k}(x_{i}^{k})>0$. By
		Theorem 4.7 in \cite{HW1}, we have $(\phi_{i}^{k}(X_{s})h(s,x^{k}_{i}))_{t\leq
			s\leq t+h}\in M^{p}_{G}(t,t+h)$, which implies that
		\[
		(\sum_{i=1}^{k_{N}}\phi_{i}^{k}(X_{s})h(s,x^{k}_{i}))_{t\leq s\leq t+h}\in
		M^{p}_{G}(t,t+h).
		\]
		Note that
		\begin{equation}%
		\begin{split}
		\left\vert \sum_{i=1}^{k_{N}}\phi_{i}^{k}(X_{s})h(s,x_{i}^{k})-h(s,X_{s}
		)\right\vert  &  \leq\sum_{i=1}^{k_{N}}\phi_{i}^{k}(X_{s})\left\vert
		h(s,x_{i}^{k})-h(s,X_{s})\right\vert +\left(  1-\sum_{i=1}^{k_{N}}\phi_{i}
		^{k}(X_{s})\right)  h(s,X_{s})\\
		&  \leq L\sum_{i=1}^{k_{N}}\phi_{i}^{k}(X_{s})|X_{s}-x_{i}^{k}|+I_{\{(B_{N}
			(0))^{c}\}}(X_{s})h(s,X_{s})\\
		&  \leq L\sum_{i=1}^{k_{N}}\phi_{i}^{k}(X_{s})\frac1k+(L|X_{s}|+h(s,0))I_{\{
			|X_{s}|\geq N\}}\\
		&  \leq L\frac1k+(L|X_{s}|+h(s,0))I_{\{ |X_{s}|\geq N\}}.
		\end{split}
		\nonumber
		\end{equation}
		Consequently, applying Theorem 4.7 in \cite{HW1} again, we get that
		\[
		\mathbb{\hat{E}}[\int_{t}^{t+h}| \sum_{i=1}^{k_{N}}\phi_{i}^{k}(X_{s}
		)h(s,x_{i}^{k})-h(s,X_{s})|^{p}ds]\rightarrow0,\ \ \ \ \text{as}
		\ k,N\rightarrow\infty,
		\]
		which implies the desired result.
	\end{proof}
	
	In the view of Lemma \ref{myq6}, we can define the solution map $\Gamma:
	U\rightarrow\Gamma(U)$ from ${M}^{\beta}_{G}(t,t+h;\mathbb{R}^{n})$ to
	${M}^{\beta}_{G}(t,t+h;\mathbb{R}^{n})$ by
	\[
	\Gamma(U):=Y^{U}, \ \ \ \ \forall U\in{M}^{\beta}_{G}(t,t+h;\mathbb{R}^{n}).
	\]
	The following lemma shows that the solution map $\Gamma$ is a contraction
	whenever $h$ is small enough.
	
	\begin{lemma}
		\label{myq9} Assume $(H_{1})-(H_{2})$ hold for some $\beta>1$. Then there is
		a constant $\delta>0$ depending only on $T,G,n, \beta$ and $L$ such that for
		any $h\in(0,\delta]$, we have
		\[
		\|{Y}^{U}-{Y}^{\bar{U}}\|_{M^{\beta}_{G}(t,t+h;\mathbb{R}^{n})}\leq\frac
		{1}{2}  \|U-\bar{U}\|_{M^{\beta}_{G}(t,t+h;\mathbb{R}^{n})},\ \ \ \ \forall U,
		\bar{U}\in M^{\beta}_{G}(t,t+h;\mathbb{R}^{n}).
		\]
		
	\end{lemma}
	
	\begin{proof}
		For each $1\leq l\leq n$, recall that $Y^{U,l}$ and $Y^{\bar{U},l}$ are the
		$l$-th components of $Y^{U}$ and $Y^{\bar{U}}$, respectively. Applying
		Proposition \ref{pro3.5} to $(Y^{U,l}-{Y}^{\bar{U},l})$ yields that
		\[
		|Y_{s}^{U,l}-{Y}_{s}^{\bar{U},l}|^{\beta}\leq C^{\star}\mathbb{\hat{E}}
		_{s}[(\int_{s}^{t+h}\hat{h}_{r}dr)^{\beta}],\ \ \forall s\in\lbrack t,t+h],
		\]
		where
		\[
		\hat{h}_{s}=|f^{l,U}(s,Y_{s}^{\bar{U},l},Z_{s}^{\bar{U},l})-f^{l,\bar{U}
		}(s,Y_{s}^{\bar{U},l},Z_{s}^{\bar{U},l})|+\sum_{i,j=1}^{d}|g_{ij}
		^{l,U}(s,Y_{s}^{\bar{U},l},Z_{s}^{\bar{U},l})-g_{ij}^{l,\bar{U}}(s,Y_{s}
		^{\bar{U},l},Z_{s}^{\bar{U},l})|
		\]
		and $C^{\star}$ is a constant depending on $T,G,\beta$ and $L$. From
		assumption $(H_{2})$ and the H\"{o}lder's inequality, we then get
		\[
		\mathbb{\hat{E}}[|Y_{s}^{U,l}-{Y}_{s}^{\bar{U},l}|^{\beta}]\leq C^{\star
		}L^{\beta}h^{\beta-1}\mathbb{\hat{E}}[\int_{t}^{t+h}|U_{r}-\bar{U}_{r}
		|^{\beta}dr],\ \ \forall s\in\lbrack t,t+h].
		\]

		Note that for each $y_{1},y_{2}\in\mathbb{R}^{n}$,
		\[
		|y_{1}-y_{2}|^{\beta}=(\sum\limits_{l=1}^{n}|y_{1}^{l}-y_{2}^{l}|^{2}
		)^{\frac{\beta}{2}}\leq n^{\max(\frac{\beta}{2}-1,0)}(\sum\limits_{l=1}
		^{n}|y_{1}^{l}-y_{2}^{l}|^{\beta}).
		\]
		Then we derive that for $s\in\lbrack t,t+h]$,
		\[
		\mathbb{\hat{E}}[|Y_{s}^{U}-{Y}_{s}^{\bar{U}}|^{\beta}]\leq n^{\max
			(\frac{\beta}{2}-1,0)}\sum\limits_{l=1}^{n}\mathbb{\hat{E}}[|Y_{s}^{U,l}
		-{Y}_{s}^{\bar{U},l}|^{\beta}]\leq n^{\max(\frac{\beta}{2},1)}C^{\star
		}L^{\beta}h^{\beta-1}\mathbb{\hat{E}}[\int_{t}^{t+h}|U_{r}-\bar{U}_{r}
		|^{\beta}dr].
		\]
		Consequently,
		\[
		\Vert{Y}^{U}-{Y}^{\bar{U}}\Vert_{M_{G}^{\beta}(t,t+h;\mathbb{R}^{n})}\leq
		|\int_{t}^{t+h}\mathbb{\hat{E}}[|Y_{s}^{U}-{Y}_{s}^{\bar{U}}|^{\beta
		}]ds|^{\frac{1}{\beta}}\leq n^{\max(\frac{1}{2},\frac{1}{\beta})}|C^{\star
		}|^{\frac{1}{\beta}}Lh\Vert U-\bar{U}\Vert_{M_{G}^{\beta}(t,t+h;\mathbb{R}
			^{n})}.
		\]

		By setting
		\begin{equation}
		\label{my104}\delta:=\min\left(  \frac{1}{2n^{\max(\frac{1}{2},\frac{1}{\beta
				})}|C^{\star}|^{\frac{1}{\beta}}L}, T\right)  ,
		\end{equation}
		we will have, for each $h\in(0,\delta]$,
		\[
		\|{Y}^{U}-{Y}^{\bar{U}}\|_{M^{\beta}_{G}(t,t+h;\mathbb{R}^{n})}\leq\frac
		{1}{2}  \|U-\bar{U}\|_{M^{\beta}_{G}(t,t+h;\mathbb{R}^{n})},
		\]
		as desired.
	\end{proof}
	
	Now we are going to use the method of Picard iteration to obtain the existence
	and uniqueness of local solutions as following, which is also crucial for the
	establishment of comparison theorem in the sequel.
	
	\begin{proof}
		[The proof of Theorem \ref{my16}]For symbol simplicity, we assume that
		$g_{ij}\equiv0$, and the proof still holds for the general case. In the
		following, unless otherwise stated, $C$ will always denote a generic constant
		which may change from line to line. The proof shall be divided into three steps
		
		\textit{1 The convergence of $Y$.} We take $\delta$ as in Lemma \ref{myq9}.
		For any given $0\leq h\leq\delta$, let $Y_{s}^{(0)}\equiv0$ and $Y_{s}
		^{(i+1)}:=Y_{s}^{Y^{(i)}}$ be iterated defined by the following
		multi-dimensional $G$-BSDE on $[t,t+h]$:
		\begin{equation}
		Y_{s}^{(i+1),l}=\zeta^{l}+\int_{s}^{t+h}f^{l,Y^{(i)}}(r,Y_{r}^{(i+1),l}
		,Z_{r}^{(i+1),l})dr-\int_{s}^{t+h}Z_{r}^{(i+1),l}dB_{r}-(K_{t+h}
		^{(i+1),l}-K_{s}^{(i+1),l}), \ 1\leq l\leq n. \label{Myeq3.3}%
		\end{equation}
		Recalling Lemma \ref{myq9}, we can get that for each $i$,
		\[
		\Vert{Y}^{(i+1)}-{Y}^{(i)}\Vert_{M_{G}^{\beta}(t,t+h;\mathbb{R}^{n})}\leq
		\frac{1}{2}\Vert{Y}^{(i)}-{Y}^{(i-1)}\Vert_{M_{G}^{\beta}(t,t+h;\mathbb{R}
			^{n})}\leq\cdots\leq\frac{1}{2^{i}}\Vert{Y}^{(1)}-{Y}^{(0)}\Vert_{M_{G}
			^{\beta}(t,t+h;\mathbb{R}^{n})},
		\]
		from which, we derive that ${Y}^{(i)}$ is a Cauchy sequence in ${M_{G}^{\beta
			}(t,t+h;\mathbb{R}^{n})}$.
		
		For each $1\leq l\leq n$, applying the estimate \eqref{my101} in Proposition
		\ref{pro3.5} to ${Y}^{(i+1),l}-{Y}^{(j+1),l}$, we get
		\[
		|{Y}_{s}^{(i+1),l}-{Y}_{s}^{(j+1),l}|^{\alpha}\leq C\mathbb{\hat{E}}_{s}%
		[(\int_{s}^{t+h}\hat{h}_{r}^{l}dr)^{\alpha}]\leq C\mathbb{\hat{E}}_{s}%
		[(\int_{t}^{t+h}\left\vert Y_{r}^{(i)}-Y_{r}^{(j)}\right\vert dr)^{\alpha}],
		\]
		where $\hat{h}_{s}^{l}=|f^{l,Y^{(i)}}(s,Y_{s}^{(i+1),l},Z_{s}^{(i+1),l}%
		)-f^{l,Y^{(j)}}(s,Y_{s}^{(i+1),l},Z_{s}^{(i+1),l})|.$ Then it follows from
		Theorem \ref{the2.10} that
		\begin{equation}%
		\begin{split}
		\mathbb{\hat{E}}[\sup_{t\leq s\leq t+h}|{Y}_{s}^{(i+1),l}-Y_{s}^{(j+1),l}%
		|^{\alpha}] &  \leq C\mathbb{\hat{E}}[\sup_{t\leq s\leq t+h}\mathbb{\hat{E}%
		}_{s}[(\int_{t}^{t+h}\left\vert Y_{r}^{(i)}-Y_{r}^{(j)}\right\vert
		dr)^{\alpha}]]\\
		&  \leq C\mathbb{\hat{E}}[(\int_{t}^{t+h}\left\vert Y_{r}^{(i)}-Y_{r}%
		^{(j)}\right\vert dr)^{\beta}])^{\frac{\alpha}{\beta}}\\
		&  \leq C\mathbb{\hat{E}}[\int_{t}^{t+h}\left\vert Y_{r}^{(i)}-Y_{r}%
		^{(j)}\right\vert ^{\beta}dr])^{\frac{\alpha}{\beta}}.
		\end{split}
		\nonumber
		\end{equation}
		Consequently, we obtain that ${Y}^{(i)}$ is a Cauchy sequence in
		${S_{G}^{\alpha}(t,t+h;\mathbb{R}^{n})}$. In particular, there exists some
		process $Y$ such that ${Y}^{(i)}$ converges to $Y$ in ${S_{G}^{\alpha
			}(t,t+h;\mathbb{R}^{n})}$. Moreover, we have
		\begin{equation}
		\Vert{Y}^{(i)}\Vert_{{S}_{G}^{\alpha}(t,t+h;\mathbb{R}^{n})}\leq
		C,\ \text{uniformly for all}\ i.\label{my105}%
		\end{equation}

		\textit{2 The existence.} Now applying the estimate (\ref{my102}) in
		Proposition \ref{pro2.9} to ${Z}^{(i+1),l}-{Z}^{(j+1),l}$, we derive that
		\begin{equation}%
		\begin{split}
		&  \mathbb{\hat{E}}[(\int_{t}^{t+h}|{Z}_{s}^{(i+1),l}-Z_{s}^{(j+1),l}%
		|^{2}ds)^{\frac{\alpha}{2}}]\\
		&  \leq C\{\Vert{Y}^{(i+1),l}-Y^{(j+1),l}\Vert_{{S}_{G}^{\alpha}
			(t,t+h;\mathbb{R}^{n})}^{\alpha}+\Vert{Y}^{(i+1),l}-Y^{(j+1),l}\Vert_{{S}
			_{G}^{\alpha}(t,t+h;\mathbb{R}^{n})}^{\frac{\alpha}{2}}[||{Y}^{(i+1),l}%
		||_{{S}_{G}^{\alpha}(t,t+h;\mathbb{R}^{n})}^{\frac{\alpha}{2}}\\
		&  +||{Y}^{(j+1),l}||_{{S}_{G}^{\alpha}(t,t+h;\mathbb{R}^{n})}^{\frac{\alpha
			}{2}}+||\int_{t}^{t+h}|f^{l,Y^{(i)}}(s,0,0)|ds||_{L_{G}^{\alpha}(\Omega
			_{t+h};\mathbb{R}^{n})}^{\frac{\alpha}{2}}+||\int_{t}^{t+h}|f^{l,Y^{(j)}%
		}(s,0,0)|ds||_{L_{G}^{\alpha}(\Omega_{t+h};\mathbb{R}^{n})}^{\frac{\alpha}{2}%
		}]\}\\
		&  \leq C\{\Vert{Y}^{(i+1),l}-Y^{(j+1),l}\Vert_{{S}_{G}^{\alpha}
			(t,t+h;\mathbb{R}^{n})}^{\alpha}+\Vert{Y}^{(i+1),l}-Y^{(j+1),l}\Vert_{{S}
			_{G}^{\alpha}(t,t+h;\mathbb{R}^{n})}^{\frac{\alpha}{2}}\},
		\end{split}
		\nonumber
		\end{equation}
		where we have used the assumptions $(H_{1})-(H_{2})$ and the estimate
		\eqref{my105} in the last inequality. Therefore, it follows from Step 1 that
		$Z^{(i)}$ is a Cauchy sequence in $H_{G}^{\alpha}(t,t+h;\mathbb{R}^{n\times
			d})$, and thus converges to some $Z\in H_{G}^{\alpha}(t,t+h;\mathbb{R}
		^{n\times d}).$
		
		For each $1\leq l\leq n$, we define
		\[
		K_{s}^{l}=Y_{s}^{l}+\int_{t}^{s}f^{l}(r,Y_{r},Z_{r}^{l})dr-\int_{t}^{s}
		Z_{r}^{l}dB_{r}-Y_{t}^{l}, \ \forall t\leq s\leq t+h.
		\]
		Then it is easy to check that $K^{(i)}_{s}$ converges to $K_{s}$ in
		$L^{\alpha}_{G}(\Omega_{s}; \mathbb{R}^{n})$ and $K\in\mathcal{A}_{G}
		^{\alpha}(t,t+h;\mathbb{R}^{n})$. This prove the existence.
		
		\textit{3 The uniqueness.} Let $(Y^{\prime},Z^{\prime},K^{\prime})$ be
		another solution of equation (\ref{my1-1}). Note that $Y,Y^{\prime}\in
		{M_{G}^{\beta}(t,t+h;\mathbb{R}^{n})}$ are both fixed points of the mapping
		$\Gamma$, then Lemma \ref{myq9} implies that $Y=Y^{\prime}$. Thus $(Y,Z,K)$
		and $(Y^{\prime},Z^{\prime},K^{\prime})$ can be viewed as the solution to the
		following $G$-BSDE:
		\[
		\label{Myeq3.15}
		\begin{split}
		\bar{Y}_{s}^{l}  &  =\zeta^{l}+\int^{t+h}_{s} f^{l}(r,Y_{r},\bar{Z}_{r}^{l})dr
		-\int^{t+h}_{s} \bar{Z}_{r}^{l} dB_{r}-(\bar{K}_{t+h}^{l}-\bar{K}_{s}
		^{l}),\ t\leq s\leq t+h,\ 1\leq l\leq n.
		\end{split}
		\]
		Therefore, applying Theorem \ref{wellposeness of G-BSDE} to components of the
		above equation yields that $(Z, K)=(Z^{\prime},K^{\prime})$. The proof is  complete.
	\end{proof}
	
	Now we are in a position to state the well-posedness for $G$-BSDE (\ref{my1})
	on the whole interval $[0,T]$.
	
	\begin{theorem}
		\label{my17}  Suppose $(H_{1})-(H_{2})$ are satisfied for some $\beta>1$. Then
		for any $1<\alpha<\beta$, the $G$-BSDE (\ref{my1})  has a unique  solution
		$(Y,Z,K)\in{S}_{G}^{\alpha}(0,T;\mathbb{R}^{n})\times H^{\alpha}%
		_{G}(0,T;\mathbb{R}^{n\times d})\times\mathcal{A}^{\alpha}_{G}(0,T;\mathbb{R}%
		^{n})$.  Moreover, $Y\in{M}_{G}^{\beta}(0,T;\mathbb{R}^{n}). $
	\end{theorem}
	
	\begin{proof}
		Choose integer $m$ large enough such that $m\delta\geq T$, where $\delta$ is
		given by Theorem \ref{my16}. Taking $h=\frac{T}{m}$ and applying Theorem
		\ref{my16}, we deduce that the $G$-BSDE (\ref{my1}) has a unique solution
		$(Y^{(m)},Z^{(m)},K^{(m)})\in{S}_{G}^{\alpha}(T-h,T;\mathbb{R}^{n})\times
		H_{G}^{\alpha}(T-h,T;\mathbb{R}^{n\times d})\times\mathcal{A}_{G}^{\alpha
		}(T-h,T;\mathbb{R}^{n})$ on the time interval $[T-h,T]$. Next we take $T-{h}$
		as the terminal time and $Y_{T-h}^{(m)}$ as the terminal condition. Applying
		Theorem \ref{my16} again, we deduce that the $G$-BSDE (\ref{my1}) admits a
		unique solution  $(Y^{(m-1)},Z^{(m-1)},K^{(m-1)})\in{S}_{G}^{\alpha
		}(T-2h,T-h;\mathbb{R} ^{n})\times H_{G}^{\alpha}(T-2h,T-h;\mathbb{R}^{n\times
			d})\times\mathcal{A}_{G}^{\alpha}(T-2h,T-h;\mathbb{R}^{n})$ on the time
		interval  $[T-2{h},T-{h}]$. Repeating this procedure, we obtain a solution
		sequence  $(Y^{(i)},Z^{(i)},K^{(i)})_{i\leq m}$. Set
		\[
		{Y}_{t}=\sum\limits_{i=1}^{m}Y_{t}^{(i)}I_{[(i-1)h,ih)}(t)+Y_{T}
		^{(m)}I_{\{T\}}(t),\ {Z}_{t}=\sum\limits_{i=1}^{m}Z_{t}^{(i)}I_{[(i-1)h,ih)}
		(t)+Z_{T}^{(m)}I_{\{T\}}(t)
		\]
		and
		\[
		K_{t}=K_{t}^{(i)}+\sum_{j=1}^{i-1}K_{jh}^{(j)},\ \text{for}\ t\in
		\lbrack(i-1)h,ih)\ \text{and}\ K_{T}=K_{T}^{(m)}+\sum_{j=1}^{m-1}K_{jh}%
		^{(j)}.
		\]
		It is easy to check that $({Y},{Z},{K})\in{S}_{G}^{\alpha}(0,T;\mathbb{R}
		^{n})\times H_{G}^{\alpha}(0,T;\mathbb{R}^{n\times d})\times\mathcal{A}
		_{G}^{\alpha}(0,T;\mathbb{R}^{n})$ is a solution to $G$-BSDE (\ref{my1}) and
		$Y\in{M}_{G}^{\beta}(0,T;\mathbb{R}^{n})$. Thus we get the existence. The
		uniqueness follows from the one on each small interval. The proof is complete.
	\end{proof}
	
	We conclude this subsection with the comparison theorem for multi-dimensional
	$G$-BSDEs (\ref{my1}). Consider the following two $G$-BSDEs on the interval
	$[0,T]$:
	\begin{equation}%
	\begin{split}
	& Y_{t}^{l} =\xi^{l}+\int_{t}^{T}f^{l}(s,Y_{s},Z_{s}^{l})ds+\int_{t}
	^{T}g_{ij}^{l}(s,Y_{s},Z_{s}^{l})d\langle B^{i},B^{j}\rangle_{s}-\int_{t}%
	^{T}Z_{s}^{l}dB_{s}-(K_{T}^{l}-K_{t}^{l}),\ 1\leq l\leq n,\\
	& \bar{Y}_{t}^{l} =\bar{\xi}^{l}+\int_{t}^{T}\bar{f}^{l}(s,\bar{Y}_{s}
	,\bar{Z}_{s}^{l})ds+\int_{t}^{T}\bar{g}_{ij}^{l}(s,\bar{Y}_{s},\bar{Z}_{s}
	^{l})d\langle B^{i},B^{j}\rangle_{s}-\int_{t}^{T}\bar{Z}_{s}^{l}dB_{s}%
	-(\bar{K}_{T}^{l}-\bar{K}_{t}^{l}),\ 1\leq l\leq n.
	\end{split}
	\nonumber
	\end{equation}
	Then we have
	
	\begin{theorem}
		\label{Myth4.1} Suppose that $f^{l}(t,y,z^{l}),\bar{f}^{l}(t,\bar{y},z^{l}
		),g^{l}_{ij}(t,y,z^{l}),\bar{g}^{l}_{ij}(t,\bar{y},z^{l})$ satisfy
		$(H_{1})-(H_{2})$  and$\ \xi,\bar{\xi}\in L_{G}^{\beta}(\Omega_{T})$ for some
		$\beta>1.$ Assume the  following conditions hold:
		
		\begin{description}

			\item[(i)] for any $1\leq l\leq n$, and for each  $t\in[0,T]$, $z^{l}%
			\in\mathbb{R}^{d}$ and $y,\bar{y}\in\mathbb{R}^{n}$ satisfying $y^{j}\geq
			\bar{y}^{j}$ for $j\neq l$ and $y^{l}=\bar{y}^{l}$, it holds that
			$f^{l}(t,y,z^{l})\geq\bar{f} ^{l}(t,\bar{y},z^{l}),[g^{l}_{ij}(t,y,z^{l}%
			)]_{i,j=1}^{d}\geq[\bar{g}^{l}_{ij}(t,\bar{y},z^{l})]_{i,j=1}^{d}$;
			
			\item[(ii)] $\xi\geq\bar{\xi}$.
		\end{description}
		
		Then we have $Y_{t}\geq\bar{Y}_{t}$ for each $t\in[0,T]$.
	\end{theorem}
	
	\begin{proof}
		For symbol simplicity, we assume $g_{ij}^{l}\equiv\bar{g}_{ij} ^{l}\equiv0$,
		and the proof for the general case is just  similar. According to the proof of
		Theorem \ref{my17},  it suffices to prove the theorem on the interval
		$[T-h,T]$, since the comparison principle  on $[0,T]$ can be obtained
		backwardly by iteration.  In the following we will prove the conclusion based
		on the Picard iteration  obtained in the proof of Theorem \ref{my16}.
		
		For each $i\in\mathbb{N}$ and $1\leq l\leq n$, let $Y^{(i),l}$ be  iterated
		defined by $Y_{t}^{(0)}\equiv0$ and
		\begin{equation}
		Y_{t}^{(i),l} =\xi^{l}+\int_{t}^{T}f^{l, Y^{(i-1)}}(s,Y_{s}^{(i),l}%
		,Z_{s}^{(i),l})ds-\int_{t}^{T}Z_{s}^{(i),l}dB_{s}-(K_{T}^{(i),l}-K_{t}%
		^{(i),l}),\ T-h\leq t\leq T.\nonumber
		\end{equation}
		Similarly, let $\bar{Y}^{(i),l}$ be iterated defined by $\bar{Y}_{t}^{(0)}
		\equiv0$ and
		\begin{equation}
		\bar{Y}_{t}^{(i),l} =\bar{\xi}^{l}+\int_{t}^{T}\bar{f}^{l,\bar{Y}^{(i-1)}%
		}(s,\bar{Y}_{s}^{(i),l} ,\bar{Z}_{s}^{(i),l})ds-\int_{t}^{T}\bar{Z}%
		_{s}^{(i),l}dB_{s}-(\bar{K} _{T}^{(i),l}-\bar{K}_{t}^{(i),l}),\ T-h\leq t\leq
		T.\nonumber
		\end{equation}
		For $i=0$, from conditions (i) and (ii), we can apply the comparison  theorem
		of one-dimensional $G$-BSDEs (Theorem \ref{the1.5}) to the components of the
		above two  equations and get
		\[
		Y_{t}^{(1)}\geq\bar{Y}_{t}^{(1)},\ \ \ \ T-h\leq t\leq T.
		\]
		Repeating the above procedure, we have
		\[
		Y_{t}^{(i)}\geq\bar{Y}_{t}^{(i)},\ \ \ \ T-h\leq t\leq T.
		\]
		Letting $i\rightarrow\infty$, according to the proof of Theorem \ref{my16},
		we  obtain finally
		\[
		Y_{t}\geq\bar{Y}_{t},\ \ \ \ T-h\leq t\leq T,
		\]
		which is the desired result.
	\end{proof}
	
	\begin{remark}
		\label{re3.1}  \upshape{
			In the routine method of Picard iteration, the iteration equation (\ref{myq4}) takes the form
			\begin{equation}\nonumber
			\begin{split}
			Y_{s}^{U,l} &  =\zeta^{l}+\int_{s}^{t+h}f^l(r,U_r,Z_{r}
			^{U,l})dr+\int_{s}^{t+h}g_{ij}^{l}(r,U_r,Z_{r}^{U,l})d\langle
			B^{i},B^{j}\rangle_{r}-\int_{s}^{t+h}Z_{r}^{U,l}dB_{r}\\
			&  \ \ \ -(K_{t+h}^{U,l}-K_{s}^{U,l}),\ \ \ 1\leq l\leq n.
			\end{split}
			\end{equation}
			But this formulation can only, in the above limit argument, obtain the comparison theorem under a  assumption stronger than (i) in Theorem \ref{Myth4.1} on generators that
			\begin{description}
				\item[(i')]
				for any $1\leq l\leq n$, and for each
				$t\in [0,T]$,  $z^l\in \mathbb{R}^{d}$ and $y,\bar{y}\in \mathbb{R}^{n}$ such that $y^{j}\geq\bar{y}^{j}$ for each $j$,   we have $f^{l}(t,y,z^{l})\geq\bar{f}	^{l}(t,\bar{y},z^{l}),[g^{l}_{ij}(t,y,z^{l})]_{i,j=1}^d\geq[\bar{g}^{l}_{ij}(t,\bar{y},z^{l})]_{i,j=1}^d$,
			\end{description}
			which,  when $n=1$, even does not contain the comparison theorem for one-dimensional $G$-BSDEs. The advantage that we adopt a new form of generators as in  (\ref{myq4}) is that we can get the comparison theorem in a general form similar to the classical case.}
		
	\end{remark}
	
	\subsection{Some useful estimates}
	
	In this section, we shall present some estimates for $Y$. For symbol
	simplicity, we will denote
	\begin{equation}
	\phi(t,y,z)=(\phi^{1}(t,y,z^{1}),\phi^{2}(t,y,z^{2}),\cdots,\phi^{n}%
	(t,y,z^{n}))^{T},\ \ \ \text{for}\ \phi=f,g_{ij}.\nonumber
	\end{equation}
	Then it holds that
	
	\begin{proposition}
		\label{Myt1} Let ${}^{\nu}\xi\in L_{G}^{\beta}(\Omega_{T};\mathbb{R}^{n})$,
		and  ${}^{\nu}f$, ${}^{\nu}g_{ij}$ satisfy $(H1)$-$(H2)$ for some  $\beta>1,$
		$\nu=1,2.$ Assume that $({}^{\nu}Y,{}^{\nu}Z,{}^{\nu}K)\in{S}_{G}^{\alpha
		}(0,T;\mathbb{R}^{n})\times H_{G}^{\alpha}(0,T;\mathbb{R}^{n\times d}
		)\times\mathcal{A}_{G}^{\alpha}(0,T;\mathbb{R}^{n})$ for some $1<\alpha
		\leq\beta$, are the solutions of equation \eqref{my1} corresponding to
		${}^{\nu}\xi$, ${}^{\nu}f$ and ${}^{\nu}g_{ij}$. Set $\hat{Y}_{t}={}^{1}Y_{t}
		-{}^{2}Y_{t}$. Then there exists a constant  $C(\alpha,n,T)$ depending only on
		$T$, $G$, $n$, $L$ and $\alpha$ such that
		\begin{equation}
		\label{Myeq4.1}|\hat{Y}_{t}|^{\alpha}\leq C(\alpha,n,T)\mathbb{\hat{E}}%
		_{t}[|\hat{\xi}|^{\alpha}+(\int_{t}^{T}\hat{h} _{s}ds)^{\alpha}],
		\end{equation}
		where $\hat{\xi}={}^{1}\xi-{}^{2}\xi$, $\hat{h}_{s}=|{}^{1}f(s,{}^{2}Y_{s}%
		,{}^{2}Z_{s})-{}^{2}f(s,{}^{2}Y_{s},{}^{2}Z_{s})|+\sum_{i,j=1} ^{d}|{}%
		^{1}g_{ij}(s,{}^{2}Y_{s},{}^{2}Z_{s})-{}^{2}g_{ij}(s,{}^{2}Y_{s},{}^{2}%
		Z_{s})|$.
	\end{proposition}
	
	\begin{proof}
		For any $1\leq l\leq n$, applying Proposition \ref{pro3.5} to $({}^{1}Y^{l}%
		-{}^{2}Y^{l})$ yields that, for each $s\in[t,T]$,
		\begin{equation}
		\label{Myeq3.5}%
		\begin{split}
		|\hat{Y}_{s}^{l}|^{\alpha}  &  \leq C(\alpha)\mathbb{\hat{E}}_{s}[|\hat{\xi
		}^{l}|^{\alpha}+(\int_{s}^{T}k_{r}^{l}dr)^{\alpha}] \leq C(\alpha
		)\mathbb{\hat{E}}_{s}[|\hat{\xi}|^{\alpha}+(\int_{s} ^{T}\hat{h}%
		_{r}dr)^{\alpha}+(\int_{s}^{T}|\hat{Y}_{r}|dr)^{\alpha}],
		\end{split}
		\end{equation}
		where $k_{r}^{l}=|{}^{1}f^{l}(r,{}^{1}Y_{r},{}^{2}Z_{r}^{l})-{}^{2}f^{l}
		(r,{}^{2}Y_{r},{}^{2}Z_{r}^{l})|+\sum_{i,j=1}^{d}|{}^{1}g_{ij}^{l}(r,{}%
		^{1}Y_{r},{}^{2}Z_{r}^{l})-{}^{2}g_{ij}^{l} (r,{}^{2}Y_{r},{}^{2}Z_{r}^{l}%
		)|$.  Denote by $C(\alpha,n,T)$ a generic constant depending on $T$, $G$, $n$,
		$L$ and $\alpha$, which may vary from line to line.  Summing up (\ref{Myeq3.5}%
		) over $l$, we conclude that
		\[
		|\hat{Y}_{s}|^{\alpha}\leq C(\alpha,n,T)(\mathbb{\hat{E}}_{s}[|\hat{\xi
		}|^{\alpha}+(\int_{t}^{T}\hat{h}_{r}dr)^{\alpha}]+\int_{s}^{T}\mathbb{\hat{E}
		}_{s}[|\hat{Y}_{r}|^{\alpha}]dr),\text{ }s\geq t.
		\]
		Taking conditional expectation $\mathbb{\hat{E}}_{t}$ on both sides, we get
		\[
		\mathbb{\hat{E}}_{t}[|\hat{Y}_{s}|^{\alpha}]\leq C(\alpha,n,T)(\mathbb{\hat
			{E}}_{t}[|\hat{\xi}|^{\alpha}+(\int_{t}^{T}\hat{h}_{r}dr)^{\alpha}]+\int%
		_{s}^{T}\mathbb{\hat{E}}_{t}[|\hat{Y}_{r}|^{\alpha}]dr),\text{ }s\geq t.
		\]
		Consequently, from Gronwall's inequality, we obtain that
		\[
		\mathbb{\hat{E}}_{t}[|\hat{Y}_{s}|^{\alpha}]\leq C(\alpha,n,T)(\mathbb{\hat
			{E}}_{t}[|\hat{\xi}|^{\alpha}+(\int_{t}^{T}\hat{h}_{r}dr)^{\alpha}]),\text{
		}s\geq t,
		\]
		which ends the proof.
	\end{proof}
	
	\begin{corollary}
		\label{Myt2} Let $\xi\in L_{G}^{\beta}(\Omega_{T};\mathbb{R}^{n})$ and
		$f,g_{ij}$ satisfy $(H1)$-$(H2)$ for some $\beta>1$. Assume that
		$(Y,Z,K)\in{S}_{G}^{\alpha}(0,T;\mathbb{R}^{n})\times H_{G}^{\alpha
		}(0,T;\mathbb{R} ^{n\times d})\times\mathcal{A}_{G}^{\alpha}(0,T;\mathbb{R}%
		^{n})$ for some  $1<\alpha\leq\beta$, is a solution of equation \eqref{my1}
		corresponding to  $\xi$, $f$ and $g_{ij}$. Then there exists a constant
		$C(\alpha,n,T)$  depending only on $T$, $G$, $L$, $n$ and $\alpha$ such that
		\begin{equation}
		\label{Myeq4.2}|{Y}_{t}|^{\alpha}\leq C(\alpha,n,T)\mathbb{\hat{E}}_{t}[|{\xi}
		|^{\alpha}+(\int_{t}^{T}h_{s}ds)^{\alpha}],
		\end{equation}
		where $h_{s}=|f(s,0,0)|+\sum_{i,j=1}^{d}|g_{ij}(s,0,0)|$.
	\end{corollary}
	
	\begin{proof}
		Taking ${}^{2}\xi\equiv0,{}^{2}f\equiv{}^{2}g_{ij}\equiv0$ in Proposition
		\ref{Myt1} and noting that ${}^{2}Y\equiv0,$ ${}^{2}Z\equiv0$,${}^{2}K\equiv
		0$, we get the desired conclusion.
	\end{proof}
	
	\begin{proposition}
		\label{Myt4.2}Let $\xi\in L_{G}^{\beta}(\Omega_{T};\mathbb{R}^{n})$ and $f$,
		$g_{ij}$ satisfy $(H1)$-$(H2)$ for some $\beta>1$. Assume that $(Y,Z,K)\in
		{S}_{G}^{\alpha}(0,T;\mathbb{R}^{n})\times H_{G}^{\alpha}(0,T;\mathbb{R}%
		^{n\times d})\times\mathcal{A}_{G}^{\alpha}(0,T;\mathbb{R}^{n})$ for some
		$1<\alpha<\beta$, is a solution of equation \eqref{my1} corresponding to $\xi
		$, $f$ and $g_{ij}$. Let any $1<\alpha<\alpha^{\prime}\leq\beta$ be given.
		Then there exists a constant $C=C(\alpha,\alpha^{\prime},n,T)$ depending only
		on $T,G,L,n,\alpha$ and $\alpha^{\prime}$ such that
		\begin{equation}
		\mathbb{\hat{E}}[\sup_{t\in\lbrack0,T]}|Y_{t}|^{\alpha}]\leq C\{(\mathbb{\hat
			{E}}[|{\xi}|^{\alpha^{\prime}}])^{\frac{\alpha}{\alpha^{\prime}}%
		}+(\mathbb{\hat{E}}[(\int_{0}^{T}h_{s}ds)^{\alpha^{\prime}}])^{\frac{\alpha
			}{\alpha^{\prime}}}\},
		\end{equation}
		where $h_{s}=|f(s,0,0)|+\sum_{i,j=1}^{d}|g_{ij}(s,0,0)|$.
	\end{proposition}
	
	\begin{proof}
		From Corollary \ref{Myt2}, we have
		\begin{equation}%
		\begin{split}
		\mathbb{\hat{E}}[\sup_{t\in\lbrack0,T]}|Y_{t}|^{\alpha}] & \leq C(\mathbb{\hat
			{E}}[\sup_{t\in\lbrack0,T]}\mathbb{\hat{E}}_{t}[|{\xi}|^{\alpha}
		]]+\mathbb{\hat{E}}[\sup_{t\in\lbrack0,T]}\mathbb{\hat{E}}_{t}[(\int_{t}
		^{T}h_{s}ds)^{\alpha}]])\\
		& \leq C(\mathbb{\hat{E}}[\sup_{t\in\lbrack0,T]}\mathbb{\hat{E}}_{t}[|{\xi
		}|^{\alpha} ]]+\mathbb{\hat{E}}[\sup_{t\in\lbrack0,T]}\mathbb{\hat{E}}%
		_{t}[(\int_{0} ^{T}h_{s}ds)^{\alpha}]]).
		\end{split}
		\nonumber
		\end{equation}
		Now the conclusion follows from Theorem \ref{the2.10}.
	\end{proof}
	
	\section{Nonlinear Feynman-Kac formula}
	
	In this section, we study the connection between the multi-dimensional
	$G$-BSDE and the  system of fully nonlinear PDEs. Indeed, we shall show that
	the solution of $G$-BSDE (\ref{my1}) associated with stochastic differential
	equation driven by $G$-Brownian motion ($G$-SDE) is the unique viscosity
	solution to a system of fully nonlinear parabolic PDEs.
	
	First, for any given $t\in\lbrack0,T]$ and $\eta\in L_{G}^{p}(\Omega
	_{t};\mathbb{R}^{k})$, $p\geq2$, we introduce the following $G$-SDEs (recall
	that we always use the Einstein's convention):
	\begin{equation}%
	\begin{cases}
	dX_{s}^{t,\eta}=b(s,X_{s}^{t,\eta})ds+h_{ij}(s,X_{s}^{t,\eta})d\langle
	B^{i},B^{j}\rangle_{s}+\sigma(s,X_{s}^{t,\eta})dB_{s},\ \ \ \ s\in\lbrack
	t,T],\\
	X_{t}^{t,\eta}=\eta,
	\end{cases}
	\label{SDE}%
	\end{equation}
	where $b(s,x),h_{ij}(s,x):[0,T]\times\mathbb{R}^{k}\rightarrow\mathbb{R}^{k}$
	and $\sigma(s,x):[0,T]\times\mathbb{R}^{k}\rightarrow\mathbb{R}^{k\times d}$
	are deterministic continuous functions satisfying:
	
	\begin{description}

		\item[$(H_{3})$] $h_{ij}=h_{ji}$ for $1\leq i,j\leq d$, and there  exists a
		positive constant $L$ such that
		\begin{equation}
		|b(t,x_{1})-b(t,x_{2})|+\sum_{i,j=1}^{d}|h_{ij}(t,x_{1})-h_{ij}(t,x_{2}%
		)|+\sum_{i=1}^{d}|\sigma(t,x_{1})-\sigma(t,x_{2})|\leq L|x_{1}-x_{2}%
		|.\nonumber
		\end{equation}
		
	\end{description}
	
	Following \cite{P10}, the equation \eqref{SDE} admits a unique $M^{p}%
	_{G}(t,T;\mathbb{R}^{k})$-solution $X^{t,\xi}$. Moreover, it holds that
	
	\begin{lemma}
		\label{the1.17} Let $\eta,\eta^{\prime}\in L_{G}^{p}(\Omega_{t};\mathbb{R}
		^{k})$ for some $p\geq2$. Then there exists a constant $C$ depending on
		$L,G,p,k$ and  $T$ such that for each $0\leq t\leq t^{\prime}\leq T$,
		
		\begin{description}

			\item[(i)] $\mathbb{\hat{E}}_{t}[\sup\limits_{s\in\lbrack t,T]}|X_{s}^{t,\eta
			}-X_{s}^{t,\eta^{\prime}}|^{p}]\leq{C}|\eta-\eta^{\prime}|^{p}$;
			
			\item[(ii)] $\mathbb{\hat{E}}_{t}[\sup\limits_{s\in\lbrack t,T]}|X_{s}
			^{t,\eta}|^{p}]\leq C(1+|\eta|^{p})$;
			
			\item[(iii)] $\mathbb{\hat{E}}_{t}[\sup\limits_{s\in\lbrack t,t^{\prime}
				]}|X_{s}^{t,\eta}-\eta|^{p}]\leq C(1+|\eta|^{p})(t^{\prime}-t)^{p/2}$.
		\end{description}
	\end{lemma}
	
	Next we consider the following $n$-dimensional $G$-BSDEs on the interval
	$[t,T]$:
	\begin{equation}%
	\begin{split}
	Y_{s}^{t,\eta;l}  &  =\varphi^{l}(X_{T}^{t,\eta})+\int_{s}^{T}f^{l}%
	(r,X_{r}^{t,\eta},Y_{r}^{t,\eta},Z_{r}^{t,\eta;l})dr+\int_{s}^{T}g_{ij}%
	^{l}(r,X_{r}^{t,\eta},Y_{r}^{t,\eta},Z_{r}^{t,\eta;l})d\langle B^{i}%
	,B^{j}\rangle_{r}\\
	&  \ \ \ -\int_{s}^{T}Z_{r}^{t,\eta;l}dB_{r}-(K_{T}^{t,\eta;l}-K_{s}%
	^{t,\eta;l}),\ \ 1\leq l\leq n,
	\end{split}
	\label{eq4.1}%
	\end{equation}
	where the deterministic continuous functions $\varphi^{l}:\mathbb{R}%
	^{k}\rightarrow\mathbb{R},f^{l},g_{ij}^{l}=g_{ji}^{l}:[0,T]\times
	\mathbb{R}^{k}\times\mathbb{R}^{n}\times\mathbb{R}^{d}\rightarrow\mathbb{R}$,
	$1\leq l\leq n,$ satisfy the following assumptions:
	
	\begin{description}

		\item[$(H_{4})$] There exists a constant $L\geq0$ such that
		\begin{equation}%
		\begin{split}
		&  |\varphi^{l}(x_{1})-\varphi^{l}(x_{2})| +|f^{l}(t,x_{1},y_{1},z^{l}
		_{1})-f(t,x_{2},y_{2},z^{l}_{2})|+\sum_{i.j=1}^{d}|g^{l}_{ij}(t,x_{1}
		,y_{1},z^{l}_{1})-g^{l}_{ij}(t,x_{2},y_{2},z^{l}_{2})|\\
		&  \leq L(|x_{1}-x_{2}|+|y_{1}-y_{2}|+|z^{l} _{1}-z^{l}_{2}|).
		\end{split}
		\nonumber
		\end{equation}
		
	\end{description}
	
	\begin{theorem}
		\label{Myth5.1} Assume that $\eta\in L_{G}^{\beta}(\Omega_{t};\mathbb{R}%
		^{k})$  for some $\beta\geq2$ and $(H_{3})$-$(H_{4})$ hold. Then, for any
		$1<\alpha<\beta$, the multi-dimensional $G$-BSDE (\ref{eq4.1}) admits a
		unique  solution $(Y^{t,\eta},Z^{t,\eta},K^{t,\eta})\in{S}_{G}^{\alpha
		}(t,T;\mathbb{R} ^{n})\times H_{G}^{\alpha}(t,T;\mathbb{R}^{n\times d}%
		)\times\mathcal{A} _{G}^{\alpha}(t,T;\mathbb{R}^{n})$.
	\end{theorem}
	
	\begin{proof}
		Since $\eta\in L_{G}^{\beta}(\Omega_{t};\mathbb{R}^{k})$, we have $X^{t,\eta
		}\in M_{G}^{{\beta}}(t,T;\mathbb{R}^{k})$ and $X_{T}^{t,\eta}\in L_{G}^{\beta
		}(\Omega_{T};\mathbb{R}^{k})$. From the Lipschitz continuity assumption of
		$\varphi,$ we can find some constant $C$ such that
		\[
		\mathbb{\hat{E}}[|\varphi(X_{T}^{t,\eta})|^{\beta}I_{\{|\varphi(X_{T}^{t,\eta
			})|>N\}}]\leq C\mathbb{\hat{E}}[(1+|X_{T}^{t,\eta}|^{\beta})I_{\{|X_{T}%
			^{t,\eta}|>{\frac{N}{C}-1}\}}]\rightarrow0,\ \text{as} \ N\rightarrow\infty,
		\]
		which together with the fact that $\varphi$ is continuous implies that
		$\varphi(X_{T}^{t,\eta})\in L_{G}^{\beta}(\Omega_{T};\mathbb{R}^{n})$. By a
		similar analysis and recalling Theorem 4.7 in \cite{HW1}, we also have
		$f^{l}(\cdot,X_{\cdot}^{t,\eta},y,z^{l}),g_{ij}^{l}(\cdot,X_{\cdot}^{t,\eta
		},y,z^{l})\in M_{G}^{\beta}(t,T)$ for each $y,z^{l}$. Now the desired result
		follows from  Theorem \ref{my17}.
	\end{proof}
	
	We also have the following regularity result on $Y^{t,\eta}$.
	
	\begin{lemma}
		\label{proA.2} Assume that $\eta$, $\eta^{\prime}\in L_{G}^{\beta}(\Omega
		_{t};\mathbb{R}^{k})$ for some $\beta\geq2$ and $(H_{3})$-$(H_{4})$ hold.
		Then  there is a constant $C$ depending on $L$, $G$, $n$, $k$ and $T$ such
		that
		\begin{equation}%
		\begin{split}
		|Y_{t}^{t,\eta}-Y_{t}^{t,\eta^{\prime}}|  &  \leq C|\eta-\eta^{\prime}|,\\
		|Y_{t}^{t,\eta}|  &  \leq C(1+|\eta|).
		\end{split}
		\nonumber
		\end{equation}
		
	\end{lemma}
	
	\begin{proof}
		We only prove the first inequality, since the other one can be obtained in a
		similar way. From Proposition \ref{Myt1}, we can find some generic constant
		$C$ depends on $L$, $G$, $n$ and $T$ that
		\begin{equation}%
		\begin{split}
		&  |Y_{t}^{t,\eta}-Y_{t}^{t,\eta^{\prime}}|^{2}\leq C\mathbb{\hat{E}}
		_{t}[|\varphi(X_{T}^{t,\eta})-\varphi(X_{T}^{t,\eta^{\prime}})|^{2}+(\int%
		_{t}^{T}\hat{h}_{s}ds)^{2}]\\
		&  \leq C\{\mathbb{\hat{E}}_{t}[|X_{T}^{t,\eta}-X_{T}^{t,\eta^{\prime}}
		|^{2}]+\int_{t}^{T}\mathbb{\hat{E}}_{t}[|X_{s}^{t,\eta}-X_{s}^{t,\eta^{\prime
		}}|^{2}]ds]\},
		\end{split}
		\label{Myeq4.5}%
		\end{equation}
		where
		\begin{equation}%
		\begin{split}
		\hat{h}_{s}  &  =|f(s,X_{s}^{t,\eta},Y_{s}^{t,\eta^{\prime}},Z_{s}
		^{t,\eta^{\prime}})-f(s,X_{s}^{t,\eta^{\prime}},Y_{s}^{t,\eta^{\prime}}
		,Z_{s}^{t,\eta^{\prime}})|\\
		&  +\sum_{i,j=1}^{d}|g_{ij}(s,X_{s}^{t,\eta},Y_{s}^{t,\eta^{\prime}}
		,Z_{s}^{t,\eta^{\prime}})-g_{ij}(s,X_{s}^{t,\eta^{\prime}},Y_{s}
		^{t,\eta^{\prime}},Z_{s}^{t,\eta^{\prime}})|.
		\end{split}
		\nonumber
		\end{equation}

		Next recalling Lemma \ref{the1.17}, we have, for $s\geq t,$
		\[
		\mathbb{\hat{E}}_{t}[|X_{s}^{t,\eta}-X_{s}^{t,\eta^{\prime}}|^{2}]\leq
		C|\eta-\eta^{\prime}|^{2}.
		\]
		Combining this with (\ref{Myeq4.5}), we get the desired result.
	\end{proof}
	
	For each $(t,x)\in\lbrack0,T]\times\mathbb{R}^{k}$, we define the following
	deterministic function
	\begin{equation}
	u(t,x):=Y_{t}^{t,x},\ \ \text{for each}\ (t,x)\in\lbrack0,T]\times
	\mathbb{R}^{k}. \label{Myeq5.1}%
	\end{equation}
	Then it follow from Lemma \ref{proA.2} that $u$ satisfies
	\begin{equation}
	|u(t,x)-u(t,x^{\prime})|\leq C|x-x^{\prime}|,\ \text{\ for all } x,x^{\prime
	}\in\mathbb{R}^{k}. \label{eq4.2}%
	\end{equation}
	In the following, we will show that $u$ is the viscosity solution of the
	following system of parabolic PDEs:
	\begin{equation}%
	\begin{cases}
	\partial_{t}u^{l}(t,x)+F^{l}(D_{x}^{2}u^{l},D_{x}u^{l},u,x,t)=0, &
	(t,x)\in(0,T)\times\mathbb{R}^{k},\ \\
	u^{l}(T,x)=\varphi^{l}(x), & x\in\mathbb{R}^{k};\ 1\leq l\leq n,
	\end{cases}
	\label{eq1.7}%
	\end{equation}
	where
	\begin{equation}%
	\begin{split}
	F^{l}(A,p,r,x,t)  &  :=G(\sigma^{T}(t,x)A\sigma(t,x)+2[\langle p,h_{ij}%
	(t,x)\rangle]_{i,j=1}^{d}+2[g_{ij}^{l}(t,x,r,\sigma^{T}(t,x)p)]_{i,j=1}%
	^{d})+\langle b(t,x),p\rangle\\
	&  +f^{l}(t,x,r,\sigma^{T}(t,x)p),\text{ for }(A,p,r,x,t)\in\mathbb{S}%
	(k)\times\mathbb{R}^{k}\times\mathbb{R}^{n}\times\mathbb{R}^{k}\times
	\lbrack0,T].
	\end{split}
	\nonumber
	\end{equation}

	Let us first give the definition of viscosity solutions of the system
	(\ref{eq1.7}).
	
	\begin{definition}
		Let $u\in C([0,T]\times\mathbb{R}^{k};\mathbb{R}^{n})$. $u$ is said to be a
		viscosity sub-solution of \eqref{eq1.7} if for each $1\leq l\leq n$,
		$u^{l}(T,x)\leq\varphi^{l}(x)$ for $x\in\mathbb{R}^{k}$, and for any
		$(t,x)\in(0,T)\times\mathbb{R}^{k}$, $\psi\in C^{1,2}((0,T)\times
		\mathbb{R}^{k})$ such that $\psi\geq u^{l} ,\psi(t,x)=u^{l}(t,x),$ we have
		\[
		\partial_{t}\psi(t,x)+F^{l}(D_{x}^{2}\psi,D_{x}\psi,u(x,t),x,t)\geq0.
		\]
		$u$ is is said to be a viscosity super-solution of \eqref{eq1.7} if for each
		$1\leq l\leq n$, $u^{l}(T,x)\geq\varphi^{l}(x)$ for $x\in\mathbb{R}^{k},$ and
		for any $(t,x)\in(0,T)\times\mathbb{R}^{k}$, $\psi\in C^{1,2}((0,T)\times
		\mathbb{R}^{k})$ such that $\psi\leq u^{l} ,\psi(t,x)=u^{l}(t,x),$ we have
		\[
		\partial_{t}\psi(t,x)+F^{l}(D_{x}^{2}\psi,D_{x}\psi,u(x,t),x,t)\leq0.
		\]
		If $u$ is both a viscosity sub-solution and a super-solution, we call $u$ a
		viscosity solution of \eqref{eq1.7}.
	\end{definition}
	
	Now we derive the following regularity result on $u$.
	
	\begin{proposition}
		\label{the1.20} The function $u(t,x):[0,T]\times\mathbb{R}^{k}\rightarrow
		\mathbb{R}^{n}$ is continuous.
	\end{proposition}
	
	\begin{proof}
		For symbol simplicity, we may assume $g_{ij}\equiv0.$ For any $t\geq0,$\ we
		define
		\[
		Y_{s}^{t,x}:=Y_{t}^{t,x},\ X_{s}^{t,x}:=X_{t}^{t,x}=x,\ K_{s} ^{t,x}%
		:=0,\ Z_{s}^{t,x}:=0,\quad\text{for each}\ s\in\lbrack0,t].
		\]
		Then $Y^{t,x}$ is the solution of the following $G$-BSDE on the interval
		$[0,T]$:
		\begin{equation}
		Y_{s}^{t,x;l}=\varphi(X_{T}^{t,x;l})+\int_{s}^{T}\tilde{f}^{l}(r,X_{r}
		^{t,x},Y_{r}^{t,x},Z_{r}^{t,x;l})dr-\int_{s}^{T}Z_{r}^{t,x;l}dB_{r}
		-(K_{T}^{t,x;l}-K_{s}^{t,x;l}),\ 1\leq l\leq n, \label{Myeq5.6}%
		\end{equation}
		where $\tilde{f}^{l}(s,X_{s}^{t,x},Y_{s}^{t,x},Z_{s}^{t,x;l})=I_{[0,t]}
		(s)f^{l}(s,X_{s}^{t,x},Y_{s}^{t,x},Z_{s}^{t,x;l}).$ For $0\leq t\leq
		t^{\prime}\leq T,$ applying Proposition \ref{Myt1}, we have
		\begin{equation}%
		\begin{split}
		&  |u(t,x)-u(t^{\prime},x)|^{2}=|Y_{0}^{t,x}-Y_{0}^{t^{\prime},x}|^{2}\\
		&  \leq C\mathbb{\hat{E}}[|\varphi(X_{T}^{t,x})-\varphi(X_{T}^{t^{\prime}
			,x})|^{2}]\\
		& +C\mathbb{\hat{E}}[|\int_{0}^{T}I_{[0,t]}(r)f^{l}(r,X_{r} ^{t,x},Y_{r}%
		^{t,x},Z_{r}^{t,x})dr-\int_{0}^{T}I_{[0,t^{\prime}]} (r)f^{l}(r,X_{r}%
		^{t,x},Y_{r}^{t,x},Z_{r}^{t,x})dr|^{2}]\\
		&  +C\mathbb{\hat{E}}[|\int_{0}^{T}I_{[0,t^{\prime}]}(r)f^{l}(r,X_{r}
		^{t,x},Y_{r}^{t,x},Z_{r}^{t,x})dr-\int_{0}^{T}I_{[0,t^{\prime}]}
		(r)f^{l}(r,X_{r}^{t^{\prime},x},Y_{r}^{t,x},Z_{r}^{t,x})dr|^{2}]\\
		&  =:I_{1}+I_{2}+I_{3.}%
		\end{split}
		\nonumber
		\end{equation}

		We will estimate the above three terms separately. First note that the
		uniqueness theorem of $G$-SDEs implies $X_{r}^{t,x}=X_{r}^{t^{\prime
			},X_{t^{\prime}}^{t,x}},$ for $r\geq t^{\prime}.$ Then by Lemma \ref{the1.17},
		we obtain
		\begin{equation}
		\mathbb{\hat{E}[}|X_{r}^{t,x}-X_{r}^{t^{\prime},x}|^{2}]=\mathbb{\hat{E}%
			[}|X_{r}^{t^{\prime},X_{t^{\prime}}^{t,x}}-X_{r}^{t^{\prime},x}|^{2}]\leq
		C\mathbb{\hat{E}[}|X_{t^{\prime}}^{t,x}-x|^{2}]\leq C(1+|x|^{2}\mathbb{)}%
		(t^{\prime}-t),\ \text{for}\ r\geq t^{\prime}.\label{eq4.5}%
		\end{equation}
		From this we get
		\[
		I_{1}\leq C\mathbb{\hat{E}}[|X_{T}^{t,x}-X_{T}^{t^{\prime},x}|^{2}]\leq
		C(1+|x|^{2})(t^{\prime}-t).
		\]
		By Proposition \ref{Myt4.2} and Lemma \ref{the1.17}, we have
		\begin{equation}
		\mathbb{\hat{E}}[\sup_{t\leq s\leq T}|Y_{s}^{t,x}|^{2}]\leq C(1+|x|^{2}%
		).\label{Myeq5.3}%
		\end{equation}
		Applying Proposition 3.5 in \cite{HJPS} to each component of (\ref{Myeq5.6})
		and using inequality (\ref{Myeq5.3}), we then deduce that
		\begin{equation}
		\mathbb{\hat{E}}[\int_{t}^{T}|Z_{s}^{t,x;l}|^{2}ds]\leq C(\mathbb{\hat{E}%
		}[\sup_{t\leq s\leq T}|Y_{s}^{t,x;l}|^{2}]+\mathbb{\hat{E}}[\int_{t}^{T}%
		|f^{l}(s,X_{s}^{t,x},Y_{s}^{t,x},0)|^{2}ds])\leq C(1+|x|^{2}).\label{Myeq5.4}%
		\end{equation}
		In spirit of inequalities (\ref{Myeq5.3}) and (\ref{Myeq5.4}), we obtain
		\begin{equation}%
		\begin{split}
		I_{2} &  =C\mathbb{\hat{E}}[|\int_{t}^{t^{\prime}}f^{l}(r,X_{r}^{t,x}%
		,Y_{r}^{t,x},Z_{r}^{t,x;l})dr|^{2}]\leq C(\mathbb{\hat{E}}[\int_{t}^{T}%
		|f^{l}(r,X_{r}^{t,x},Y_{r}^{t,x},Z_{r}^{t,x;l})|^{2}dr])(t^{\prime}-t)\\
		&  \leq C(1+|x|^{2})(t^{\prime}-t).
		\end{split}
		\nonumber
		\end{equation}
		Applying estimate (\ref{eq4.5}) again, we have
		\begin{equation}%
		\begin{split}
		I_{3} &  \leq C\mathbb{\hat{E}}[\int_{0}^{T}I_{[0,t^{\prime}]}(r)|f^{l}%
		(r,X_{r}^{t,x},Y_{r}^{t,x},Z_{r}^{t,x})dr-f^{l}(r,X_{r}^{t^{\prime},x}%
		,Y_{r}^{t,x},Z_{r}^{t,x})|^{2}dr]\\
		&  \leq C\mathbb{\hat{E}}[\int_{0}^{T}I_{[0,t^{\prime}]}(r)|X_{r}^{t,x}%
		-X_{r}^{t^{\prime},x}|^{2}dr]\leq C(1+|x|^{2}\mathbb{)}(t^{\prime}-t).
		\end{split}
		\nonumber
		\end{equation}

		Consequently, it holds that
		\[
		|u(t,x)-u(t^{\prime},x)|\leq C(1+|x|)(t^{\prime}-t)^{1/2},
		\]
		which together with inequality (\ref{eq4.2}) indicate the desired result.
	\end{proof}
	
	We will also need the following lemma, whose proof is similar to that of
	Theorem 4.4 in \cite{HJPS1} and so is omitted.
	
	\begin{lemma}
		\label{theA.4} Assume $\eta\in L_{G}^{\beta}(\Omega_{t};\mathbb{R}^{k})$ for
		some $\beta\geq2$. Then
		\begin{equation}
		u(t,\eta)=Y_{t}^{t,\eta}. \label{Myeq5.2}%
		\end{equation}
		
	\end{lemma}
	
	We now state the main result of this section.
	
	\begin{theorem}
		\label{Myth4.3}  Assume $(H_{3})$-$(H_{4})$ hold. Then $u$ is the unique
		viscosity  solution of the backward parabolic equation system (\ref{eq1.7}).
	\end{theorem}
	
	\begin{proof}
		We first prove that $u$ is a viscosity solution of (\ref{eq1.7}).  We only
		show that $u$ is a viscosity sub-solution, and the proof of  sup-solution is similar.
		
		Let $1\leq l\leq n$ be given. Assume $\psi\in C^{1,2}((0,T)\times
		\mathbb{R}^{k})$ such that
		$\psi\geq u^{l}$, $\psi(t,x)=u^{l}(t,x)$ for some $(t,x)\in(0,T)\times
		\mathbb{R}^{k}$. Since $u$ is Lipschitz continuous and from a standard
		argument, we can assume that $\psi$ is continuously differential  up to the
		third order and the corresponding derivatives of order from 1 to 3 are
		bounded. We need to prove that
		\begin{equation}
		\partial_{t}\psi(t,x)+F^{l}(D_{x}^{2}\psi(x,t),D_{x}\psi(x,t),u(x,t),x,t)\geq
		0. \label{Myeq5.11}%
		\end{equation}

		For $t^{\prime}\geq t$, let $(\tilde{Y},\tilde{Z},\tilde{K})$ be the solution
		the following one-dimensional $G$-BSDE on $[t,t^{\prime}]:$
		\begin{equation}
		\label{Myeq4.9}%
		\begin{split}
		\tilde{Y}_{s}  &  =\psi(t^{\prime},X_{t^{\prime}}^{t,x})+\int_{s}^{t^{\prime}
		}f^{l}(r,X_{r}^{t,x},Y_{r}^{t,x;1},\cdots,Y_{r}^{t,x;l-1},\tilde{Y}_{r}
		,Y_{r}^{t,x;l+1},\cdots,Y_{r}^{t,x;n},\tilde{Z}_{r})dr\\
		&  +\int_{s}^{t^{\prime}}g_{ij}^{l}(r,X_{r}^{t,x},Y_{r}^{t,x;1},\cdots
		,Y_{r}^{t,x;l-1},\tilde{Y}_{r},Y_{r}^{t,x;l+1},\cdots,Y_{r}^{t,x;n},\tilde
		{Z}_{r})d\langle B^{i},B^{j}\rangle_{r}\\
		&  -\int_{s}^{t^{\prime}}\tilde{Z}_{r}dB_{r}-(\tilde{K}_{t^{\prime}}-\tilde
		{K}_{s}).
		\end{split}
		\end{equation}
		We first note that, from the $G$-It\^{o}'s formula (see Section III of
		\cite{P10}),
		\begin{equation}%
		\begin{split}
		{{\psi}}(t^{\prime},X_{t^{\prime}}^{t,x})  &  ={{\psi}}(s,X_{s}^{t,x}
		)+\int_{s}^{t^{\prime}}\partial_{t}\psi(r,X_{r}^{t,x})dr+\int_{s}^{t^{\prime}
		}\langle D_{x}\psi(r,X_{r}^{t,x}),b(r,X_{r}^{t,x})\rangle dr\\
		&  +\int_{s}^{t^{\prime}}\langle D_{x}\psi(r,X_{r}^{t,x}),h_{ij}
		(r,X_{r}^{t,x})\rangle d\langle B^{i},B^{j}\rangle_{r}+\int_{s}^{t^{\prime}
		}\langle\sigma^{T}(r,X_{r}^{t,x})D_{x}\psi(r,X_{r}^{t,x}),dB_{r}\rangle\\
		&  +\frac{1}{2}\int_{s}^{t^{\prime}}(\sigma^{T}(r,X_{r}^{t,x})D^{2}_{x}%
		\psi(r,X_{r}^{t,x}) \sigma(r,X_{r}^{t,x}))_{ij}d\langle B^{i},B^{j}\rangle
		_{r},\quad s\in\lbrack t,t^{\prime}].
		\end{split}
		\nonumber
		\end{equation}
		Moreover, by the uniqueness theorem of multi-dimensional $G$-BSDEs and Lemma
		\ref{theA.4}, we have
		\[
		Y^{t,x}_{s}=Y^{s,X^{t,x}_{s}}_{s}=u(s,X^{t,x}_{s}),\ \ \text{for}\ s\geq t.
		\]
		Thus, if we denote $\hat{Y}_{s}=\tilde{Y}_{s}-\psi(s,X_{s}^{t,x})$, $\hat
		{Z}_{s}=\tilde{Z}_{s}-(\sigma^{T}D_{x}\psi)(s,X_{s}^{t,x})$, $\hat{K}
		_{s}=\tilde{K}_{s}$. Then $(\hat{Y},\hat{Z},\hat{K})$ is the solution of the
		following $G$-BSDE on $[t,t^{\prime}]:$
		\begin{equation}
		\label{Myeq4.12}%
		\begin{split}
		\hat{Y}_{s}=  &  \int_{s}^{t^{\prime}}e^{l}(r,X_{r}^{t,x},\hat{Y}_{r},\hat
		{Z}_{r})dr+\int_{s}^{t^{\prime}}k_{ij}^{l}(r,X_{r}^{t,x},\hat{Y}_{r},\hat
		{Z}_{r})d\langle B^{i},B^{j}\rangle_{r}\\
		&  -\int_{s}^{t^{\prime}}\hat{Z}_{r}dB_{r}-(\hat{K}_{t^{\prime}}-\hat{K}
		_{s}){.}%
		\end{split}
		\end{equation}
		where
		\begin{equation}%
		\begin{split}
		e^{l}(s,x,y,z)  &  =f^{l}(s,x,u^{1}(s,x),\cdots,u^{l-1}(s,x),y+\psi
		(s,x),u^{l+1}(s,x),\cdots,u^{n}(s,x),z^{l}+(\sigma^{T}D_{x}\psi)(s,x))\\
		&  +\partial_{t}\psi(s,x)+\langle b(s,x),D_{x}\psi(s,x)\rangle,
		\end{split}
		\nonumber
		\end{equation}
		\begin{equation}%
		\begin{split}
		k_{ij}^{l}(s,x,y,z)  &  =g_{ij}^{l}(s,x,u^{1}(s,x),\cdots,u^{l-1}
		(s,x),y+\psi(s,x),u^{l+1}(s,x),\cdots,u^{n}(s,x),z^{l}+(\sigma^{T}D_{x}
		\psi)(s,x))\\
		&  +\langle D_{x}\psi(s,x),h_{ij}(s,x)\rangle+\frac{1}{2}(\sigma^{T}D_{x}
		^{2}\psi\sigma)_{ij}(s,x).
		\end{split}
		\nonumber
		\end{equation}

		We now consider a deterministic approximation of the above $G$-BSDE
		(\ref{Myeq4.12}).  Let $\ (\bar{Y} ,\bar{Z},\bar{K})$ be the solution of the
		following $G$-BSDE on $[t,t^{\prime}]:$
		\begin{equation}
		\bar{Y}_{s}=\int_{s}^{t^{\prime}}e^{l}(r,x,\bar{Y}_{r},\bar{Z}_{r})dr+\int%
		_{s}^{t^{\prime}}k_{ij}^{l}(r,x,\bar{Y}_{r},\bar{Z}_{r})d\langle B^{i}
		,B^{j}\rangle_{r}-\int_{s}^{t^{\prime}}\bar{Z}_{r}dB_{r}-(\bar{K}_{t^{\prime}%
		}-\bar{K}_{s}). \label{Myeq5.8}%
		\end{equation}
		We first take $\bar{Z}=0$. Note that
		\begin{equation}%
		\begin{split}
		\bar{Y}_{s}  &  =\int_{s}^{t^{\prime}}[e^{l}(r,x,\bar{Y}_{r},0)+2G([k_{ij}
		^{l}(r,x,\bar{Y}_{r},0)]_{i,j=1}^{d})]dr\\
		&  +\int_{s}^{t^{\prime}}k_{ij}^{l}(r,x,\bar{Y}_{r},0)d\langle B^{i}
		,B^{j}\rangle_{r}-\int_{s}^{t^{\prime}}2G([k_{ij}^{l}(r,x,\bar{Y}
		_{r},0)]_{i,j=1}^{d})dr.
		\end{split}
		\nonumber
		\end{equation}
		Thus if we take $\bar{Y}_{s}$ as the solution of the following ODE:
		\begin{equation}
		\label{Myeq4.10}\bar{Y}_{s}=\int_{s}^{t^{\prime}}[e^{l}(r,x,\bar{Y}%
		_{r},0)+2G([k_{ij} ^{l}(r,x,\bar{Y}_{r},0)]_{i,j=1}^{d})]dr,
		\end{equation}
		and define
		\[
		\bar{K}_{s}:=\int_{t}^{s}k_{ij}^{l}(r,x,\bar{Y}_{r},0)d\langle B^{i}
		,B^{j}\rangle_{r}-\int_{t}^{s}2G([k_{ij}^{l}(r,x,\bar{Y}_{r},0)]_{i,j=1}
		^{d})dr,
		\]
		then it is easy to see that $(\bar{Y},0,\bar{K})\ $is the solution of
		(\ref{Myeq5.8}). Applying Proposition \ref{pro3.5} and Lemma \ref{the1.17},
		we  have
		\[
		|\hat{Y}_{t}-\bar{Y}_{t}|^{2}\leq C\mathbb{\hat{E}}_{t}[(\int_{t}^{t^{\prime}
		}\hat{h}_{r}dr)^{2}]\leq C(\int_{t}^{t^{\prime}}\mathbb{\hat{E}}_{t}[|\hat
		{h}_{r}|^{2}]dr)(t^{^{\prime}}-t)\leq C(1+|x|^{2})(t^{\prime}-t)^{3},
		\]
		where
		\[
		\hat{h}_{r}=|e^{l}(r,X_{r}^{t,x},\bar{Y}_{r},0)-e^{l}(r,x,\bar{Y}_{r}
		,0)|+\sum_{i,j=1}^{d}|k_{ij}^{l}(r,X_{r}^{t,x},\bar{Y}_{r},0)-k_{ij}
		^{l}(r,x,\bar{Y}_{r},0)|.
		\]
		Therefore,
		\begin{equation}
		|\hat{Y}_{t}-\bar{Y}_{t}|\leq C(1+|x|)(t^{\prime}-t)^{\frac{3}{2} }.
		\label{Myeq4.6}%
		\end{equation}

		Moreover, in view of the assumption on $\psi$ and applying Theorem
		\ref{the1.5} to $G$-BSDEs (\ref{Myeq4.9}) and
		\begin{equation}%
		\begin{split}
		Y_{s}^{t,x;l}  &  =u^{l}(t^{\prime},X_{t^{\prime}}^{t,x})+\int_{s}^{t^{\prime}
		}f^{l}(r,X_{r}^{t,x},Y_{r}^{t,x},Z_{r}^{t,x;l})dr+\int_{s}^{t^{\prime}}
		g_{ij}^{l}(r,X_{r}^{t,x},Y_{r}^{t,x},Z_{r}^{t,x;l})d\langle B^{i},B^{j}
		\rangle_{s}\\
		&  \ \ \ -\int_{s}^{t^{\prime}}Z_{r}^{t,x;l}dB_{r}-(K_{t^{\prime}}^{t,x;l}
		-K_{s}^{t,x;l}),\text{ }t\leq s\leq t^{\prime},
		\end{split}
		\nonumber
		\end{equation}
		we get
		\[
		\tilde{Y}_{t}\geq u^{l}(t,x).
		\]
		This implies
		\[
		\hat{Y}_{t}\geq0.
		\]
		Combining this with (\ref{Myeq4.6}), we obtain
		\[
		(t^{\prime}-t)^{-1}\bar{Y}_{t}\geq-C(1+|x|)(t^{\prime}-t)^{1/2}.
		\]
		That is,
		\begin{equation}
		\label{Myeq3.12}(t^{\prime}-t)^{-1}\int_{t}^{t^{\prime}}[e^{l}(r,x,\bar{Y}%
		_{r},0)+2G([k_{ij} ^{l}(r,x,\bar{Y}_{r},0)]_{i,j=1}^{d})]dr\geq
		-C(1+|x|)(t^{\prime}-t)^{1/2}.
		\end{equation}
		On the other hand, note that from (\ref{Myeq4.10}),
		\[
		|\bar{Y}_{s}|\leq C\int_{s}^{t^{\prime}}(1+|\bar{Y}_{r}|)dr\leq C((t^{\prime
		}-t)+\int_{s}^{t^{\prime}}|\bar{Y}_{r}|)dr,\ \ s\in[t,t^{\prime}].
		\]
		which implies, by the Gronwall's inequality,
		\[
		|\bar{Y}_{s}|\leq C(t^{\prime}-t),\ \ s\in[t,t^{\prime}].
		\]
		Therefore, letting $t^{\prime}\rightarrow t$ in (\ref{Myeq3.12}), we get
		\[
		e^{l}(t,x,0,0)+2G([k_{ij} ^{l}(r,x,0,0)]_{i,j=1}^{d})\geq0,
		\]
		which is just (\ref{Myeq5.11}) and this means that $u$ is a viscosity sub-solution.
		
		Analysis similar to the proof of Theorem 5.1 in \cite{PPR} (see also
		\cite{BBP,BH,BL}) shows that there exists at most one solution to the system
		(\ref{eq1.7}) in the class of continuous functions with polynomially growth.
		The proof is complete.
	\end{proof}

\bigskip

\noindent\textbf{Acknowledgement}: The author is grateful to Dr. Mingshang Hu
 and Dr. Falei Wang  for their fruitful discussions. The author would also like to thank the anonymous referee for the valuable
 comments and constructive suggestions which improved the presentation of this manuscript.

\end{document}